\documentclass[10pt,reqno]{amsart}
\usepackage{amscd,amsmath,amssymb,amsthm,amsfonts}
\usepackage{lipsum}
\usepackage{xpatch}  
\xpatchcmd{\paragraph}{\normalfont}{{\normalfont\bfseries}}{}{}
\usepackage{cancel}
\usepackage{color,xcolor}
\usepackage{graphics}
\usepackage{inputenc}
\usepackage{fontenc}

\usepackage[english]{babel}

\newcommand{\myspace}{\qquad\qquad\qquad}

\newcommand{\cA}{{\mathcal A}}
\newcommand{\cB}{{\mathcal B}}
\newcommand{\cD}{{\mathcal D}}

\newcommand{\cH}{{\mathcal H}}
\newcommand{\cI}{{\mathcal I}}
\newcommand{\cK}{{\mathcal K}}
\newcommand{\cL}{{\mathcal L}}

\newcommand{\cQ}{{\mathcal Q}}

\newcommand{\cU}{{\mathcal U}}

\newtheorem{theorem}{Theorem}[section]
\newtheorem{lemma}[theorem]{Lemma}
\newtheorem{proposition}[theorem]{Proposition}
\newtheorem{remark}[theorem]{Remark}

\newtheorem{assumptions}[theorem]{Assumptions}

\newtheorem{corollary}[theorem]{Corollary}

\newtheorem{problem}[theorem]{Problem}

\numberwithin{equation}{section}

\usepackage{mathrsfs}
\usepackage[margin=37mm]{geometry}

\makeatletter
\@namedef{subjclassname@2020}{\textup{2020} Mathematics Subject Classification}
\makeatother

\date{}

\begin{document}


\title[LQ control of parabolic-like evolutions with memory of the inputs]{Linear quadratic control of parabolic-like evolutions with memory of the inputs} 

\author{Paolo Acquistapace}
\address{Paolo Acquistapace, Universit\`a di Pisa ({\em Ret.}), 
Dipartimento di Matematica, Largo Bruno~Pontecorvo 5, 56127 Pisa, ITALY 
}
\email{paolo.acquistapace(at)unipi.it}

\author{Francesca Bucci}
\address{Francesca Bucci, Universit\`a degli Studi di Firenze,
Dipartimento di Matematica e Informatica,
Viale Morgagni 65, 50139 Firenze, ITALY
}
\email{francesca.bucci(at)unifi.it}



\begin{abstract}
A study of the linear quadratic~(LQ)~control problem on a finite time interval for a model equation 
in Hilbert spaces which comprehends the memory of the inputs was performed recently by the authors.
The outcome included a closed-loop representation of the unique optimal control, 
along with the derivation of a related coupled system of three quadratic (operator) equations which is shown to be well-posed.
Notably, in the absence of memory the above elements -- namely, formula and system -- reduce to the known feedback formula and single differential Riccati equation, respectively.
In this work we take the next natural step, and prove the said results for a class
of evolutions where the control operator is no longer bounded.
These findings appear to be the first ones of their kind; furthermore, they extend the classical theory of the LQ problem and Riccati equations for~parabolic~partial differential equations.     
\end{abstract}

\maketitle

{\small

\hfill {\sl To Giuseppe (Beppe) Da Prato}

\hfill{\sl great scientist and friend}

\hfill{\sl with esteem and affection}

\hfill{\sl remembering past times and his sharing of the love of mathematics}

}


\section{Introduction and main result}
The question of attaining a full synthesis of the optimal solution in the finite time horizon optimal control problem with quadratic functionals for important classes of linear partial differential equations (PDE)
subject to boundary actions has been extensively studied in the last forty-five years or so.
Given that the unique ({\em open-loop}) minimizer does exist, the actual sought-after goal is to attain a representation of the optimal control in {\em closed-loop} form first of all, and then to attain -- possibly for suitable subclasses of functionals only -- that the (linear, bounded) operator that occurs in the feedback formula solves uniquely a Riccati equation.
Thus proving the well-posedness of appropriate Riccati equations is a crucial step in the study
of the linear quadratic (LQ)~problem for evolutionary PDE.
Theoretical findings and significant PDE illustrations are provided in \cite{las-trig-redbooks} dealing with parabolic-like and hyperbolic-like evolutions. The papers \cite{ac-bu-las_2005,ac-bu-las_2013,ac-bu-uniqueness_2023} develop a theory suited to deal with a class of control systems which encompasses
distinct coupled systems comprising both parabolic and hyperbolic PDE components
(such as e.g.~some which describe thermoelastic systems, acoustic-structure and fluid-structure interactions).

It is well known that diverse physical phenomena such as viscoelasticity or heat conduction as well as the evolution of population dynamics may bring about model equations where the presence of memory terms accounts for the influence of the past values of one or more variables in play.
Consider now a simple, albeit relevant, example such as a linear heat equation with finite memory
in a bounded domain $\Omega\subset\mathbb{R}^d$, supplemented with initial and boundary data, to wit, 
\begin{equation*}
\begin{cases}
w_t(t,x) = \Delta w(t,x) +\displaystyle\int_0^t N(\sigma) w(t-\sigma,x)\,d\sigma & \text{in $(0,T)\times \Omega=:Q_T$}
\\
w(t,x)=u(t,x) & \text{on $(0,T)\times\partial\Omega=:\Sigma_T$}
\\
w(0,x) = w_0(x) & \text{in $\Omega$.}
\end{cases}
\end{equation*}

\noindent
At the outset, let us think of the function $u=u(t,x)$ as a given boundary datum.
Then, by using the renowned Fattorini-Balakrishnan method to attain an abstract (re)formulation of the boundary value problem, one easily finds the integro-differential equation
\begin{equation*}
w'=A(w-Du)+\displaystyle\int_0^t N(\sigma)A[w(t-\sigma)-Du(t-\sigma)]\,d\sigma\,, \qquad t\in (0,T]
\end{equation*}
in the unknown $w(t):=w(t,\cdot)$ and with $u(t):=u(t,\cdot)$, where $A$ is the realization of the Laplacian $\Delta$ in $L^2(\Omega)$ with homogeneous Dirichlet boundary condition, while
$D$ is the (so called) Dirichlet mapping -- namely, the map that associates to a boundary datum its harmonic extension in the interior of the domain $\Omega$. 
Setting $B=-AD$, thus with 
\begin{equation*}
B\colon L^2(\partial\Omega)\longrightarrow [\cD(A^*)]'
\end{equation*} 
an {\em unbounded} operator, one arrives at the integro-differential equation
\begin{equation} \label{e:full-eq}
w'=Aw+Bu+\int_0^t N(t-\sigma)\big[Aw(\sigma)+Bu(\sigma)\big]\,d\sigma\,, \qquad t\in (0,T]\,;
\end{equation}
see e.g. \cite[Section~3.1]{pandolfi-book}.
It is apparent in \eqref{e:full-eq} and is important to emphasise that the very same operator control $B$ pops up inside the convolution integral.  

\smallskip
With the function $u(t,x)$ now interpreted as a boundary input on $\partial\Omega$, allowed to vary in an appropriate class of admissible controls, we note that the controlled integro-differential equation \eqref{e:full-eq} poses various technical challenges:

\begin{itemize}
\item
the realization $A$ of the differential operator occurs in the convolution term (and yet MacCamy's trick may help to remove it),

\item
the control operator $B$ is unbounded,

\item 
(last but not least) the past values of both the dynamics variable $w$ and the control $u$ 
influence the evolution.
\end{itemize}

When it comes to the LQ problem for evolution equations with memory, a clear analog of the Riccati equations
which are connected to the minimization problem in the memoryless case was unavailable until very recently.
With focus on a simple integro-differential equation in $\mathbb{R}^d$, the work \cite{pandolfi-memory_2018} infers a coupled system of three quadratic (matrix) equations associated with the optimal control problem; solving it provides
the matrices that occur in the feedback formula in a univocal manner.
This result has clarified a question not entirely solved in \cite{pritchard-you_1996} and remained open
for more than twenty years. 
A subsequent extension to a more general class of control systems and to tracking-type functionals 
is found in \cite{pandolfi-memory_2024}.

Spurred by the advances in a finite-dimensional context, in the absence of this type of findings in the PDE literature and with various distinct technical challenges to be tackled, the authors pursued a strategy where the difficulties are taken one at a time.
In a first work, that is \cite{ac-bu-memory1_2024}, we adopted a variational and Riccati-based approach to the LQ problem for the control system
\begin{equation*}
w'=Aw+Bu+\displaystyle\int_0^t k(t-\sigma)w(\sigma)\,d\sigma\,, \qquad t\in (0,T]\,,
\end{equation*}
where in comparison to \eqref{e:full-eq} the dyamics operator $A$ is absent in the integral term and the past values of the control are not involved; still, we solved an open problem until then.

Subsequently, in \cite{ac-bu-JOTA} we focused on the LQ problem for a control system where the memory of the control function
$u$ is brought into the picture, whereas the memory of $w$ is neglected, that is
\begin{equation} \label{e:system-start}
w'=Aw+Bu+\int_0^t k(t-\sigma) Bu(\sigma)\,d\sigma\,, \qquad t\in (0,T]\,.
\end{equation}
We note that in both \cite{ac-bu-memory1_2024} and \cite{ac-bu-JOTA} it is assumed that $B$ is a {\em bounded} operator; however, the respective outcomes are distinct and the studies overcome specific technical hurdles.

\smallskip
In this article we expand the reach of our work to deal with the latter model equation 
in the case when $B$ is unbounded, while $A$ is the generator of strongly continuous semigroup in a Hilbert space $\cH$, which in addition is analytic -- these two features being consistent with the full 
integro-differential equation \eqref{e:full-eq} -- and ascertain the findings of \cite{ac-bu-JOTA}
in this more complicated setting. 

As we have done successfully in our recent works \cite{ac-bu-memory1_2024,ac-bu-JOTA}, we adapt to the problem at hand the general line of argument carried out in the study of the LQ problem for relevant classes of memoryless infinite-dimensional control systems that describe boundary value problems for PDE \cite{las-trig-redbooks}.
The major steps of this path involve

\begin{itemize}
\item[-]
the existence of a unique minimizer (the open-loop optimal control),
\item[-]
the optimality condition, which in particular brings about 
\item[-] 
an operator $P(t)$ which enters the feedback formula $\hat{u}(t)=-B^*P(t)\hat{w}(t)$ as well as
the optimal cost;  
\item[-]
whether $P(t)$ does solve the differential Riccati equation (RE) corresponding to the optimal control problem, which establishes the property of {\em existence} for the RE; then the key issue is 
\item[-] 
{\em uniqueness} for the RE,
\end{itemize}
thereby achieving the closed-loop synthesis of the optimal control.
The precise functional-analytic setting, the main results and an outline of the paper are provided in the next subsections.
 
We conclude this introductory part including some bibliographical references (the list is by no means exhaustive).
Suggested monographs are \cite{renardy-etal_1987}, \cite{pruess_1993}, \cite{pandolfi-book}, along with the references therein.
Still in the context of optimal control for deterministic evolution equations with memory, for more general frameworks than the LQ one for the controlled dynamics and/or the functionals to be minimized -- in particular, semilinear PDE and/or non-quadratic costs -- see \cite{cannarsa-etal_2013} and \cite{casas-yong-memory_2023}, where the optimal strategies are characterized via first- and second-order optimality conditions, respectively.
Although there does not seem to be an overlapping with our present and earlier work, we point out the following works pertaining to stochastic model equations: 
\cite{wang-t_2018}, \cite{bonaccorsi-confortola_2020}, \cite{abi-etal_2021a,abi-etal_2021b}
\cite{han-etal_2023}, \cite{wang-h-etal_2023}, \cite{hamaguchi-wang_2024}.

Due to space limitations, it is not possible to give an account of the various contributions
to the other great questions of control theory for integro-differential PDE such as reachability, controllability, unique continuation, observability and inverse problems via Carleman estimates, stability and uniform decay rates.
We remark, however, that most of the aforementioned studies concern PDE with infinite
memory of the evolution, which is not our case. 


\subsection{Setup}
Consider the control system \eqref{e:system-start}, supplemented with an initial condition 
$w(0)=w_0\in \cH$; the function $u(\cdot)$ -- having a role of a control action -- varies in
$\cU:=L^2(0,T;U)$.
The function spaces $\cH$ and $U$, the operators $A$ and $B$, the kernel $k$ are assumed to satisfy the following properties.

 
\begin{assumptions}[Abstract setup] \label{a:ipo_0} 
Let $\cH$, $U$ be separable complex Hilbert spaces. It is assumed that
\begin{itemize}
\item[A1.]
The linear operator $A\colon \cD(A)\subset \cH \to \cH$ is the infinitesimal generator of a strongly continuous semigroup $\{e^{At}\}_{t\ge 0}$ on $\cH$, which is also analytic; hence, the fractional powers $(\lambda_0-A)^\alpha$, $\alpha\in (0,1)$, are well-defined for some $\lambda_0>0$; 

\smallskip

\item[A2.] 
the control operator $B$ satisfies $B\in \cL(U,[\cD(A^*)]')$, and 
\begin{equation}
\exists \gamma\in (0,1)\colon \quad (\lambda_0-A)^{-\gamma}B\in \cL(U,\cH) ;
\end{equation}

\smallskip

\item[A3.]
the kernel $k(\cdot)$ satisfies $k\in L^2(0,T;\cL(\cH))$, along with the commutativity property $e^{tA}k=ke^{tA}$.
\end{itemize}

\end{assumptions} 

\begin{remark}
{\begin{rm} (On the values $\lambda_0$ and $\gamma$). 
\end{rm}}
We will set $\lambda_0=0$ and denote the fractional powers $A^\alpha$ (in place of 
$(-A)^\alpha$) throughout this work for simplicity.
Furthermore, in the sequel focus will be placed on the values $\gamma>1/2$ of the parameter in A2. of the Assumptions~\ref{a:ipo_0}.
On one side, the range $(1/2,1)$ for the values of $\gamma$ brings about a worse regularity of relevant functions/operators; and in addition, motivation for the consideration of this range comes from the optimal boundary control of the heat equation with memory and Dirichlet boundary input, for which we have $\gamma=3/4+\epsilon$, $\epsilon\in (0,1/4)$ -- besides, in fact, $\lambda_0=0$.
\end{remark}

\smallskip
It is natural to introduce the concept of {\em mild} solution to the Cauchy problems associated with the control system \eqref{e:system-start}, namely, the one that corresponds to a given initial datum $w_0\in \cH$ at the initial time $t=0$ and to a control action $u(\cdot)\in L^2(0,T;U)$, which
reads as
\begin{equation} \label{e:mild-sln}
w(t)=e^{At} w_0 +\int_0^t e^{A(t-q)}Bu(q)\,dq +\int_0^t e^{A(t-q)} \int_0^q k(q-p)Bu(p)\,dp\,dq
\end{equation}
and whose regularity properties (in time and space) will be clarified combining 

\begin{itemize}

\item
the well known regularity results pertaining to the (so called) {\em input-to-state} map $L_s$ defined
for any $s\in [0,T)$ by 
\begin{equation} \label{e:L_s}
L_s\colon u(\cdot) \longmapsto (L_su)(t) =\int_s^t e^{A(t-r)}Bu(r)\,dr\,,
\end{equation}
as well as the ones for its adjoint
\begin{equation} \label{e:L_s*}
L_s^*\colon f(\cdot) \longmapsto (L_s^*f)(t)=\int_t^T B^*e^{A^*(\sigma-t)}f(\sigma)\,d\sigma
\end{equation}
(see e.g. \cite[Proposition 3.4]{ac-bu-las_2013} or Proposition~\ref{p:Ls} at the end), 
which are both key in the memoryless case, along with

\item 
the regularity properties of the novel operator $H_s$ (as well as $\cK_s$) brought about by the memory,
discussed and proved in the next section; see Proposition~\ref{e:regularity-newop}.
\end{itemize}


To the model equation \eqref{e:system-start} we associate the following quadratic functional over a
given (finite) time interval $[0,T]$:
\begin{equation} \label{e:cost}
J(u)=J_T(u,w_0)=\int_0^T \left(\|Cw(t)\|_{\cH}^2 + \|u(t)\|_U^2\right)dt\,, 
\end{equation}
where the weighting operator $C$ simply satisfies 

\begin{equation} \label{e:ipo_1}
\hspace{-8cm} \text{(A4)} \qquad\myspace C\in \cL(\cH)\,.
\end{equation}

\noindent
The simplified notation $J(u)$ should be self-explanatory and will be used throughout. 

The optimal control problem is formulated in the usual classical way.


\begin{problem}[The optimal control problem] \label{p:problem-0}
Given $w_0\in \cH$, seek a control function $\hat{u}(\cdot)=\hat{u}(\cdot,0,w_0)$ which minimizes the functional \eqref{e:cost} overall $u\in L^2(0,T;U)$, where $w(\cdot)$ is the mild solution to \eqref{e:system-start} (given by \eqref{e:mild-sln}) corresponding to the control function $u(\cdot)$ and with initial datum $w_0$ (at time $0$). 
\end{problem}


\subsection{Main results} 
A foundational point of our line of argument -- known in the literature as the dynamic programming
approach, dating back to the work of Richard E.\,Bellman in the fifties  -- is the embedding of the optimal control problem in a family of similar optimization problems, depending on suitable parameters, here the initial time $s\in [0,T)$ -- besides the initial {\em state}, whose actual structure will be clarified below, see \eqref{e:initial-data}.

Following an approach pursued e.g. in the works \cite{pandolfi-memory_2018}, \cite{ac-bu-memory1_2024},
\cite{pandolfi-memory_2024} and \cite{ac-bu-JOTA}, with the initial time $s$ allowed to vary in the interval $[0,T)$, we consider as initial data the elements 
\begin{equation} \label{e:initial-data}
X_0=\begin{cases} w_0 & s=0
\\[1mm]
\begin{pmatrix}w_0
\\[1mm] 
\eta(\cdot)
\end{pmatrix} & 0<s<T
\end{cases}
\end{equation}
where $\eta(\cdot)$ is a given function in $L^2(0,T;U)$; we accordingly define the state space as
\begin{equation}\label{e:state_space}
Y_s:= \begin{cases}
\cH & s=0
\\
\cH\times L^2(0,s;U) & 0<s<T\,.
\end{cases}
\end{equation}

Consistently, the mild solution to the control system \eqref{e:system-start} supplemented with the initial datum $X_0$ (defined in \eqref{e:initial-data}) at time $s$ is rewritten as
\begin{equation} \label{e:mild-sln_s}
w(t)=e^{A(t-s)} w_0 +\big[\big(L_s+H_s)u\big](t)+ \cK_s \eta(t)\,,
\end{equation}
with the operators $L_s$, $H_s$ and $\cK_s$ defined by
\begin{subequations}
\begin{align}
L_s u(t)\equiv(L_s u)(t)&=\int_s^t e^{A(t-q)}Bu(q)\,dq
\\[1mm]
H_s u(t)\equiv(H_s u)(t)&=\int_s^t e^{A(t-\sigma)} \int_s^\sigma k(\sigma-q)Bu(q)\,dq\,d\sigma
\label{e:Hs_0} 
\\[1mm]
\cK_s \eta(t)\equiv(\cK_s \eta)(t)&=\int_s^t e^{A(t-\sigma)} \int_0^s k(\sigma-q)B\eta(q)\,dq\,d\sigma
\notag\\
&= \int_0^s e^{A(t-\sigma)} \int_s^t k(\sigma-q)B\eta(q)\,dq\,d\sigma\,.
\label{e:cKs_0}
\end{align}
\end{subequations}
By setting
\begin{equation} \label{e:lambda}
\lambda(t,q,s):=\int_s^t e^{A(t-\sigma)} k(\sigma-q)B\,d\sigma\,, \qquad 0\le s\le q\le t< T\,,
\end{equation} 
the terms $H_s u(t)$ and $\cK_s u(t)$ (in \eqref{e:Hs_0} and \eqref{e:cKs_0}, respectively) are rewritten readily and neatly as follows:

\begin{align}
H_s u(t) &=\int_s^t \lambda(t,q,q)u(q)\,dq\,,
\label{e:Hs_1}
\\[1mm]
\cK_s \eta(t) &=\int_0^s \lambda(t,q,s)\eta(q)\,dq\,.
\label{e:cKs_1}
\end{align}

We introduce the family of functionals
\begin{equation} \label{e:cost_s}
J_s(u)=J_{T,s}(u,X_0)=\int_s^T \big(\|Cw(t)\|_\cH^2 + \|u(t)\|_U^2\big)\,dt\,;
\end{equation}
the relative optimal control problem is formulated in a natural way.


\begin{problem}[Parametric optimal control problem] \label{p:problem_s}
Given $X_0\in Y_s$, seek a control function $\hat{u}=\hat{u}(\cdot,s,X_0)$ which minimizes the functional \eqref{e:cost_s} overall $u\in L^2(s,T;U)$, where $w(\cdot)$ -- given by \eqref{e:mild-sln_s} -- is the solution to \eqref{e:system-start} corresponding to a control function $u(\cdot)$ and with initial datum $X_0$ (at time $s$). 
\end{problem}


In the following result we gather the principal findings of this work, namely, the specific representation of the unique optimal control in closed-loop form, but also that the three linear and bounded operators which occur in the said formula do solve {\em uniquely} a system of coupled quadratic equations.
 
\begin{theorem}[Main results] \label{t:main}
With reference to the optimal control problem \eqref{e:mild-sln_s}-\eqref{e:cost_s}, under the Assumptions~\ref{a:ipo_0} and the hypothesis \eqref{e:ipo_1}, the following statements are valid for any $s\in [0,T)$.

\begin{enumerate}

\item[\textbf{S1.}] 
For each $X_0\in Y_s$ there exists a unique optimal pair $(\hat{u}(\cdot,s,X_0),\hat{w}(\cdot,s,X_0))$
which satisfies 
\begin{equation} \label{e:better-reg}
\hat{u}(\cdot,s,X_0)\in C([s,T],U)\,, \quad \hat{w}(\cdot,s,X_0)\in C([s,T],\cH)\,.
\end{equation}

\smallskip

\item[\textbf{S2.}]
There exist three linear bounded operators, denoted by $P_0(s)$, $P_1(s,p)$, and $P_2(s,p,q)$ -- 
defined in terms of the optimal evolution and of the data of the problem(see the expressions \eqref{e:riccati-ops} and \eqref{e:riccati-ops_2})  
--, such that the optimal cost is given by  
\begin{equation} \label{e:optimal-cost_1}
\begin{split}
\qquad J_s(\hat{u}) &=\big(P_0(s) w_0,w_0\big)_{\cH}
+ 2\text{Re}\, \int_0^s \big(P_1(s,p)\eta(p),w_0\big)_{\cH} \,dp
\\[1mm]
& + \int_0^s \!\!\int_0^s \big(P_2(s,p,q)\eta(p),\eta(q)\big)_U\,dp\,dq
\equiv \big(P(s)X_0,X_0\big)_{Y_s}\,.
\end{split}
\end{equation}
$P_0(s)$ and  $P_2(s,p,q)$ are self-adjoint and non-negative operators in the respective functional
spaces $H$ and $L^2(0,s;U)$; in addition, it holds 
\begin{equation*}
P_2(s,p,q)=P_2(s,q,p)\,.
\end{equation*}

\smallskip

\item[\textbf{S3.}] 
The optimal control admits the following representation:
\begin{equation} \label{e:feedback_2}
\begin{split}
\hat{u}(t,s,X_0)&=-\big[B^*P_0(t)+P_1(t,t)^*\big]\hat{w}(t,s,X_0)
\\
& \qquad -\int_0^t \big[B^*P_1(t,p)+P_2(t,p,t)\big]\theta(p)\,dp\,,
\end{split}
\end{equation}
with
\begin{equation*}
\theta(\cdot)=
\begin{cases} \eta(\cdot) & \text{in $[0,s)$}
\\
\hat{u}(\cdot,s,X_0) & \text{in $[s,t)$}
\end{cases}
\end{equation*}
and the operators $P_i$ are given by the formulas \eqref{e:riccati-ops_2} (originally, \eqref{e:riccati-ops}), $i\in \{0,1,2\}$. 

\smallskip

\item[\textbf{S4.}] 
The operators $P_0(t)$, $P_1(t,p)$, $P_2(t,p,q)$ -- as from \textbf{S3.} --
satisfy the following coupled system of equations, for every $t\in[0,T)$, $p,q\in [0,t]$, and for any $x,y\in \cD(A)$, $v,u\in U$:
\begin{equation} \label{e:DRE}
\hspace{3mm}
\begin{cases}
&\frac{d}{dt}\big(P_0(t)x,y\big)_{\cH} +\big(P_0(t)x,Ay\big)_{\cH} 
+ \big(Ax,P_0(t)y\big)_{\cH} + \big(C^*Cx,y\big)_{\cH}  
\\[1mm] 
&\myspace - \big([B^*P_0(t)+P_1(t,t)^*]x,[B^*P_0(t)+P_1(t,t)^*]y\big)_U=0 
\\[2mm]
&\frac{\partial}{\partial t} \big(P_1(t,p)v,y\big)_{\cH} + \big(P_1(t,p)v,Ay\big)_{\cH} 
+ \big(k(t-p)Bv,P_0(t)y\big)_{\cH} 
\\[1mm] 
&\qquad\quad
- \big([B^*P_1(t,p)+P_2(t,p,t)]v,[B^*P_0(t)+P_1(t,t)^*]y\big)_U =0
\\[2mm]
& \frac{\partial}{\partial t}\big( P_2(t,p,q)u,v\big)_U
+\big(P_1(t,p)u,k(t-q)Bv\big)_{\cH}
+\big(k(t-p)Bu,P_1(t,q)v\big)_{\cH}
\\[1mm]
&\qquad\quad  -\big([B^*P_1(t,p)+P_2(t,p,t)]u,[B^*P_1(t,q)+P_2(t,q,t)]v\big)_U=0
\end{cases}
\end{equation}
with final conditions
\begin{equation} \label{e:final}
P_0(T)=0\,, \; P_1(T,p)=0\,, \; P_2(T,p,q)=0\,.
\end{equation} 

\smallskip

\item[\textbf{S5.}] 
There exists a unique triplet $(P_0(t),P_1(t,p),P_2(t,p,q))$ that solves the coupled system \eqref{e:DRE} and fulfils the final conditions \eqref{e:final}, within the class of linear bounded operators (in the respective spaces), the former and the latter being self-adjoint and non-negative.  
\end{enumerate}

\end{theorem}

\begin{remark}
We remark at the outset that despite the fact that the statements S1.-S5 of Theorem~\ref{t:main} are the same as those in \cite[Theorem~1, Section~1.1]{ac-bu-JOTA}, additional technical challenges are present and need to be overcome at several key steps of the respective proofs, in view of the unboundedness of the control operator $B$.
\end{remark}


\subsection{An outline of the paper}
The paper is organized as follows.
In the next section we pinpoint certain regularity properties of the operators brought about by the memory and which occur in the representation of the mild solutions to the control system.
These preliminary results allow to accurately assess the regularity in time of the $w$-component of the state variable (see Corollary~\ref{c:mild-slns}); in addition, they are utilized in the study of the optimization problem.
 
Section~\ref{s:feedback} focuses on the statements S1., S2.~and~S3. of Theorem~\ref{t:main}.
In order to prove S1., our starting point is once again the optimality condition.
We note here that unlike the case when the control operator is bounded, the continuity in time of the optimal solution $\hat{u}(\cdot)$ cannot be taken for granted; see Corollary~\ref{c:reg_coppia}. 
The three operators $P_i$ ($i=0,1,2$) which are building blocks of the quadratic form representing the optimal cost are singled out in Proposition~\ref{p:riccati-ops_0}.
A distinct reformulation of these operators achieved in Lemma~\ref{l:key} is called for in order to 
ascertain that they actually occur in a representation of the optimal control in closed-loop form, which eventually will be~\eqref{e:feedback_2}.

The proof of the statements S4. and S5. of Theorem~\ref{t:main}, namely, of the fact that the operators $P_i$ ($i=0,1,2$) solve uniquely the coupled system of three quadratic differential equations \eqref{e:DRE} is laid out in Section~\ref{s:wellposedRE}. 
In comparison with our earlier work \cite{ac-bu-JOTA}, the analysis needs to be supplemented with additional preliminary steps, as a consequence of the unboundedness of the control operator $B$; see Proposition~\ref{p:riccati-ops-continui} and Proposition~\ref{p:bounded-gains} -- the latter addressing the delicate issue of boundedness of the gain operators --, in turn based on the novel Lemma~\ref{l:regularity} and Lemma~\ref{l:increm_Z}.
 
A short Appendix recalls a few instrumental results pertaining to the regularity of the input-to-state map in the memoryless parabolic case, and to convolution integrals.


\section{Prerequisite regularity results} \label{s:regularity}
Aiming to establish the regularity of any mild solution \eqref{e:mild-sln_s} and since the regularity of the map $L_s$ is well known, we pinpoint in this section the regularity properties of the operators $H_s$, $\cK_s$ (brought about by the memory) and the respective adjoints $H_s^*$, $\cK_s^*$.
We need to explore the one of $\lambda(t,q,s)$ as well as to produce appropriate estimates of certain differences, to accomplish this; the outcomes of this analysis are stated in two separate Lemmas. 
(These are not trivial, due to the presence of the unbounded operator $B$.)


\begin{lemma} \label{l:lemma1}
Let the Assumptions~\ref{a:ipo_0} with $\gamma>\frac{1}{2}$ be valid, and 
let $\lambda(t,q,s)$ be the operator defined in \eqref{e:lambda}.
Then, with $r=\frac{2}{2\gamma-1}$ we have for $0\le  q \le s \le t \le T$
\begin{subequations}
\begin{align}
&\|\lambda(\cdot,q,s)\|_{L^r(s,T;\cL(U,\cH))}\lesssim \|k\|_{L^2(0,T;\cL(\cH))},
\label{e:stima1}
\\[1mm]
&\|\lambda(t,\cdot,s)\|_{L^r(0,s;\cL(U,\cH))}\lesssim \|k\|_{L^2(0,T;\cL(\cH))},
\label{e:stima2}
\\[1mm]
q\longmapsto \|\lambda(t,q,q)\|_{\cL(U,\cH)}\in L^r(s,T), & \; \textrm{with}
\notag
\\[1mm]  
&\|\lambda(t,\cdot,\cdot)\|_{L^r(s,T;\cL(U,\cH))}\lesssim \|k\|_{L^2(0,T;\cL(H))}.
\label{e:stima3}
\end{align}
\end{subequations}

\end{lemma} 


\begin{proof}
In order to establish \eqref{e:stima1}, we rewrite $\lambda(t,q,s)$ as
\begin{equation*}
\lambda(t,q,s)=\int_s^t A^\gamma e^{A(t-\sigma)} k(\sigma-q)[A^{-\gamma}B]\,d\sigma\,,
\end{equation*}  
which gives
\begin{equation*}
\|\lambda(t,q,s)\|_{\cL(U,\cH)}\le C\,\int_s^t \frac{1}{(t-\sigma)^\gamma} \|k(\sigma-q)\|_{\cL(\cH)}\,d\sigma
\end{equation*}  
for some positive constant $C$.
Then \eqref{e:stima1} follows by using \cite[Theorem~383]{hardy-etal}.
Concerning \eqref{e:stima2}, we write
\begin{equation*}
\int_0^s \|\lambda(t,q,s)\|_{\cL(U,\cH)}^r\,dq \lesssim \int_0^s \left[\int_s^t \frac1{(t-\sigma)^\gamma}\|k(\sigma-q)\|_{\cL(\cH)}\,d\sigma\right]^r dq\,,
\end{equation*}
and setting $p=\sigma-q$, $\tau=s-q$, we get 
\begin{equation*}
\int_0^s \|\lambda(t,q,s)\|_{\cL(U,\cH)}^r\,dq \lesssim \int_0^s \left[\int_\tau^{\tau+t-s} \frac1{(\tau+t-s-p)^\gamma} \|k(p)\|_{\cL(\cH)}\,dp\right]^r d\tau\,.
\end{equation*}
Using again \cite[Theorem~383]{hardy-etal} we conclude that
\begin{equation*}
\int_0^s \|\lambda(t,q,s)\|_{\cL(U,\cH)}^r\,dq \lesssim \|k\|_{L^2(0,T;\cL(\cH))}^r\,.
\end{equation*}
The estimate \eqref{e:stima3} is shown in a similar way.
\end{proof}


We now estimate the increments of $\lambda(t,q,s)$ with respect to each variable. 

\begin{lemma}\label{e:reg_lambda}
Let the Assumptions~\ref{a:ipo_0} with $\gamma>\frac{1}{2}$ be valid, and let $\lambda(t,q,s)$ be the operator defined in \eqref{e:lambda}. 
Then:

\begin{enumerate}
\item[(i)] for $T\ge t\ge \tau \ge s$ 
\begin{equation} \label{e:diff_lambda1}
\|\lambda(t,\cdot,s)-\lambda(\tau, \cdot,s)\|_{L^2(0,s;\cL(U,\cH))} \le C (t-\tau)^\theta, \quad 0<\theta<1-\gamma;
\end{equation}
\item[(ii)] for $T >s\ge q\ge p$
\begin{equation} \label{e:diff_lambda2}
\|\lambda(\cdot,q,s)-\lambda(\cdot, p,s)\|_{L^2(s,T;\cL(U,\cH))} \le 
C \|k(\cdot)-k(\cdot+q-p)\|_{L^2(0,T-q;\cL(U,\cH))};
\end{equation}
\item[(iii)] for $T>s > \sigma\ge q$
\begin{equation} \label{e:diff_lambda3}
\|\lambda(\cdot,q,s)-\lambda(\cdot,q,\sigma)\|_{L^2(s,T;\cL(U,\cH))} \le C(s-\sigma)^{1-\gamma} .
\end{equation}
\end{enumerate}

\end{lemma}

\begin{proof} 
\textbf{(i)} For $T\ge t\ge \tau \ge s \ge q$ we write
\begin{equation*}
\begin{split}
&\lambda(t,q,s)-\lambda(\tau,q,s) =\int_\tau^t e^{A(t-\sigma)}k(\sigma-q)B\,d\sigma 
+\int_s^\tau \Big[\int_{\tau-\sigma}^{t-\sigma}Ae^{Ar}\,dr\Big]k(\sigma-q)B\,d\sigma
\\
& \quad = \int_\tau^t A^\gamma e^{A(t-\sigma)}k(\sigma-q)[A^{-\gamma}B]\,d\sigma 
+ \int_s^\tau \Big[\int_{\tau-\sigma}^{t-\sigma}A^{1+\gamma}e^{Ar}\,dr\Big]k(\sigma-q)[A^{-\gamma}B]\,d\sigma\,,
\end{split}
\end{equation*}  
so
\begin{equation}\label{e:differ1}
\begin{split}
\|\lambda(t,q,s)-\lambda(\tau,q,s)\|_{\cL(U,\cH)}
&\le C\int_\tau^t\frac{1}{(t-\sigma)^\gamma}\|k(\sigma-q)\|_{\cL(\cH)}\,d\sigma
\\
& \quad +C\int_s^\tau\int_{\tau-\sigma}^{t-\sigma}\frac{dr}{r^{1+\gamma}}\|k(\sigma-q)\|_{\cL(\cH)}\,d\sigma.
\end{split}
\end{equation}
The $L^2(0,s;\cL(U,\cH))$-norm of the terms in the right-hand side of \eqref{e:differ1} is estimated 
via the H\"older inequality, to find 
\[
\left[\int_\tau^t\frac{1}{(t-\sigma)^\gamma}\|k(\sigma-q)\|_{\cL(\cH)}\,d\sigma\right]^2 \le  C(t-\tau)^{2(1-\gamma)}\int_\tau^t \frac{1}{(t-\sigma)^\gamma}\|k(\sigma-q)\|_{\cL(\cH)}^2d\sigma;
\]
thus, by \cite[Theorem 383]{hardy-etal} we get
\[
\int_0^s \left[\int_\tau^t\frac{1}{(t-\sigma)^\gamma}\|k(\sigma-q)\|_{\cL(\cH)}\,d\sigma\right]^2 dq \le C(t-\tau)^{2(1-\gamma)} \|k\|_{L^2(0,T;\cH)}^2.
\]
As for the second term in the right-hand side of \eqref{e:diff_lambda1}, we have, for any $\theta\in \,(0,1-\gamma)\,$, 
\[
\left[\int_s^\tau\int_{\tau-\sigma}^{t-\sigma}\frac{dr}{r^{1+\gamma}}\|k(\sigma-q)\|_{\cL(\cH)}\,d\sigma\right]^2 \le \left[\int_s^\tau \frac{(t-\tau)^\theta}{(\tau-\sigma)^{\gamma+\theta}}\|k(\sigma-q)\|_{\cL(\cH)}d\sigma\right]^2
\]
and we deduce as before
\[
\int_0^s \left[\int_s^\tau\int_{\tau-\sigma}^{t-\sigma}\frac{dr}{r^{1+\gamma}}\|k(\sigma-q)\|_{\cL(\cH)}\,d\sigma\right]^2 \le C(t-\tau)^{2\theta} \|k\|_{L^2(0,T;\cH)}^2.
\]
\textbf{(ii)} Similarly, for $t\ge s > \sigma \ge q$,
\[
\lambda(t,q,s)-\lambda(t,q,\sigma) = -\int_s^\sigma A^\gamma e^{A(t-r)}k(r-q)(A^{-\gamma}B)\,dr,
\]
which yields
\begin{equation}\label{e:differ2}
\|\lambda(t,q,s)-\lambda(t,p,s)\|_{\cL(U,\cH)}\le C\int_s^t \frac1{(t-\sigma)^\gamma} \|k(\sigma-q)-k(\sigma-p)\|_{\cL(\cH)}\,d\sigma;
\end{equation}
proceeding as above we obtain the desired estimate.

\noindent
\textbf{(iii)} For $T\ge t \ge s \ge \sigma \ge q$, the estimate 
\begin{equation} \label{e:differ3}
\|\lambda(t,q,s)-\lambda(t,q,\sigma)\|_{\cL(U,\cH)}
\le C\int_s^\sigma \frac1{(t-r)^\gamma}\|k(r-q)\|_{\cL(\cH)}\,dr
\end{equation}
holds true.
The estimate for the $L^2(s,T;\cL(U,\cH))$-norm of the right-hand side of \eqref{e:differ3} is quite similar to the preceding ones, hence we omit it.
\end{proof}


We now use the previous results to pinpoint the regularity of the maps brought about by the memory.

\begin{proposition} \label{e:regularity-newop}
Let the Assumptions~\ref{a:ipo_0} on the operators $A$, $B$, $k(\cdot)$ hold true, with 
$\gamma>\frac{1}{2}$.
Then, the following regularity results pertain to the operators $H_s$, $\cK_s$ defined by \eqref{e:Hs_0} and \eqref{e:cKs_0} and to their respective adjoints $H_s^*$, $\cK_s^*$:
\begin{align}
& H_s\in \cL(L^2(s,T;U),C^\alpha([s,T],\cH)),
\label{e:reg-H}\\[1mm]
& \cK_s\in \cL(L^2(0,s;U),C^\alpha([0,s],\cH)),
\label{e:reg-K}\\[1mm]
& H_s^*\in \cL(L^2(s,T;\cH),C([s,T],U)),
\label{e:reg-H*}\\[1mm]
&\cK_s^*\in \cL(L^2(0,s;\cH),C([0,s],U)),
\label{e:reg-K*}
\end{align}
for any $\alpha< 1-\gamma$.
\end{proposition}
\begin{proof}
\textbf{1.} 
Starting from \eqref{e:Hs_1}, with $k\in L^2(0,T;\cL(\cH))$ and given $u\in L^2(s,T;U)$, we have
\begin{equation*} 
\begin{split}
\|H_su(t)\|_\cH &\le \int_s^t \|\lambda(t,q,q)\|_{\cL(U,\cH)}\|u(q)\|_U\,dq
\\
& \le \Big(\int_s^t \|\lambda(t,q,q)\|_{\cL(U,\cH)}^2\Big)^{1/2}\,
\Big(\int_s^t \|u(q)\|_U^2\,dq\Big)^{1/2}\le C
\end{split}
\end{equation*}
for some constant $C>0$ as a consequence of the H\"older inequality, which proves 
$H_su\in L^\infty(0,T;\cH)$.
The above basic regularity can be actualy enhanced: in order to prove the claimed H\"older continuity property \eqref{e:reg-H}, we evaluate $\|H_su(t)-H_su(\tau)\|_\cH$ for given $t,\tau\ge s\ge 0$ as follows, using the estimate established in \eqref{e:differ2}:
\begin{equation} \label{e:holderH}
\begin{split}
& \|H_su(t)-H_su(\tau)\|_\cH \\
& \qquad \le \Big\|\int_\tau^t \lambda(t,q,q)u(q)\,dq\Big\|_\cH
+ \Big\|\int_s^\tau \big[\lambda(t,q,q)-\lambda(\tau,q,q)\big]u(q)\,dq\Big\|_\cH
\\[1mm]
&\qquad \le \|\lambda(t,\cdot,\cdot)\|_{L^r(\tau,t;\cL(U,\cH))}(t-\tau)^{1/2-1/r}\|u\|_{L^2(s,T;U)}
\\
& \qquad \qquad + c \int_s^\tau\int_q^\tau \frac{(t-\tau)^\gamma}{(t-\sigma)^\gamma (\tau-\sigma)^\gamma}
\|k(\sigma-q)\|_{\cL(\cH)} \|u(q)\|_U\,d\sigma\,dq
\\[1mm]
& \qquad \le c(t-\tau)^{1/2-1/r}\|k\|_{L^2(0,T;\cL(H))}\|u\|_{L^2(s,T;U)}
\\
& \qquad \qquad 
+ \underbrace{c (t-\tau)^\alpha\int_s^\tau\int_q^\tau \frac{1}{(\tau-\sigma)^{\gamma+\alpha}}
\|k(\sigma-q)\|_{\cL(\cH)} \|u(q)\|_U\,d\sigma\,dq}_{s_2}\,,
\end{split}
\end{equation}
where $r=2/(2\gamma-1)$ and $\alpha\in (0,1-\gamma)$. 
We move on with the estimate of the second summnd $s_2$ in the right hand side, to find
\begin{equation} \label{e:bound4s2}
\begin{split}
s_2 &\le c (t-\tau)^\alpha \int_s^\tau\int_s^\sigma \|k(\sigma-q)\|_{\cL(\cH)} \|u(q)\|_U\,dq
\frac{1}{(\tau-\sigma)^{\gamma+\alpha}}\,d\sigma
\\[1mm]
&\le c (t-\tau)^\alpha \int_s^\tau\Big[\int_s^\sigma \|k(\sigma-q)\|_{\cL(\cH)}^2\,dq\Big]^{1/2}
\|u\|_{L^2(s,T;U)}\frac{1}{(\tau-\sigma)^{\gamma+\alpha}}\,d\sigma
\\[1mm]
&\le c (t-\tau)^\alpha \|k\|_{L^2(0,T;\cL(\cH))}\|u\|_{L^2(s,T;U)}\,.
\end{split}
\end{equation}
Once we have the estimate \eqref{e:bound4s2}, we return to \eqref{e:holderH}: thus, 
since  $1/2-1/r=1-\gamma$ while $\alpha<1-\gamma$, we have $\min\{1/2-1/r,\alpha\}=\alpha$ 
which establishes $H_su\in C^\alpha([s,T],\cH)$ (for any $\alpha <1-\gamma$).

\smallskip
\noindent
\textbf{2.} Similarly, starting from \eqref{e:cKs_1}, with $k\in L^2(0,T;\cL(\cH))$ and given $\eta\in L^2(0,s;U)$, we first find $\cK_s\eta\in L^\infty(0,s;\cH)$.
As before, the regularity can be actually improved to H\"older continuity: with $\alpha<1-\gamma$
we get
\begin{equation} \label{e:holderK}
\begin{split}
&\|\cK_s\eta(t)-\cK_s\eta(\tau)\|_{\cH} 
=\Big\|\int_0^s \big[\lambda(t,q,s)-\lambda(\tau,q,s)\big]\eta(q)\,dq\Big\|_\cH
\\[1mm]
& \qquad \le c \int_0^s\int_\tau^t \frac{1}{(t-\sigma)^\gamma}\|k(\sigma-q)\|_{\cL(H\c)} \|\eta(q)\|_U\,d\sigma\,dq
\\
& \qquad \qquad +c \int_0^s\int_s^\tau \frac{(t-\tau)^\gamma}{(t-\sigma)^\gamma (\tau-\sigma)^\gamma}
\|k(\sigma-q)\|_{\cL(\cH)} \|\eta(q)\|_U\,d\sigma\,dq
\\[1mm]
& \qquad \le c \int_\tau^t \frac{1}{(t-\sigma)^\gamma}\int_0^s\|k(\sigma-q)\|_{\cL(\cH)} \|\eta(q)\|_U\,dq\,d\sigma
\\
& \qquad \qquad +c \int_s^\tau \frac{(t-\tau)^\alpha}{(\tau-\sigma)^{\gamma+\alpha}}
\int_0^s\|k(\sigma-q)\|_{\cL(\cH)} \|\eta(q)\|_U\,dq\,d\sigma
\\[1mm]
& \qquad \le c \int_\tau^t \frac{1}{(t-\sigma)^\gamma}\,d\sigma \|k\|_{L^2(0,T;\cL(\cH)} \|\eta\|_{L^2(0,s;U)}
\\
&\qquad \qquad + c (t-\tau)^\alpha \int_s^\tau \frac{1}{(\tau-\sigma)^{\gamma+\alpha}}
\|k\|_{L^2(0,T;\cL(\cH)} \|\eta\|_{L^2(0,s;U)}
\\[1mm]
& \qquad \le c \big[(t-\tau)^{1-\gamma}+(t-\tau)^\alpha\big]\|k\|_{L^2(0,T;\cL(\cH)} \|\eta\|_{L^2(0,s;U)}\,.
\end{split}
\end{equation} 
As $\alpha<1-\gamma$, $\cK_s\eta\in C^\alpha([s,T],\cH)$.

\smallskip
\textbf{3.} 
The adjoint operator of $H_s$ is found computing -- for any pair $u(\cdot)\in L^2(s,T;U)$ and
$z(\cdot)\in L^2(s,T;\cH)$ -- the scalar product
\begin{equation*}
\begin{split}
\big(H_su,z\big)_{L^2(s,T;\cH)} 
&=\int_s^T\Big(\int_s^t e^{A(t-\sigma)}\int_s^\sigma k(\sigma-q)Bu(q)\,dq\,d\sigma,z(t)\Big)_{\cH}\,dt 
\\
&=\int_s^T\int_s^t\int_s^\sigma \big(u(q),B^* k(\sigma-q)^*e^{A^*(t-\sigma)}z(t)\big)_U\,dq\,d\sigma\,dt\,;
\end{split}
\end{equation*}
by using the Fubini-Tonelli Theorem we find
\begin{equation*}
\begin{split}
\big(H_su,z\big)_{L^2(s,T;\cH)}  
&= \int_s^T\int_s^t\int_q^t \big(u(q),B^* k(\sigma-q)^*e^{A^*(t-\sigma)}z(t)\big)_U\,d\sigma\,dq\,dt
\\
&= \int_s^T\int_q^T\int_q^t \big(u(q),B^* k(\sigma-q)^*e^{A^*(t-\sigma)}z(t)\big)_U\,d\sigma\,dt\,dq
\\
&= \int_s^T\Big(u(q),\int_q^T\int_q^t B^* k(\sigma-q)^*e^{A^*(t-\sigma)}z(t)\,d\sigma\,dt\Big)_U\,dq\,,
\end{split}
\end{equation*}
which establishes
\begin{equation*}
H_s^* z(q)=\int_q^T\int_q^t B^* k(\sigma-q)^*e^{A^*(t-\sigma)}z(t)\,d\sigma\,dt\,, \qquad q\in [s,T]\,.
\end{equation*}

Then, since owing to the Assumptions~\ref{a:ipo_0} $k^*$ commutes with $e^{A^*\cdot}$, we get
\begin{equation*}
\begin{split}
\|H_s^* z(q)\|_U 
&\le c\,\int_q^T\int_q^t \|k(\sigma-q)^*\|_{\cL(\cH)}\frac{1}{(t-\sigma)^\gamma}\|z(t)\|_U\,d\sigma\,dt
\\
& = c\,\int_q^T\Big[\int_q^t \frac{1}{(t-\sigma)^\gamma}\|k(\sigma-q)^*\|_{\cL(\cH)}\,d\sigma\Big]
\|z(t)\|_U\,dt\,
\\ 
&\le c
\left[\int_q^T\Big[\int_q^t \frac{1}{(t-\sigma)^\gamma}
\|k(\sigma-q)^*\|_{\cL(H)}\,d\sigma\Big]^2\,dt\right]^{1/2}
\Big[\int_q^T\|z(t)\|_U^2\,dt\Big]^{1/2}
\\[1mm]
&\le c\,\|k\|_{L^2(0,T;\cL(\cH))} \|z\|_{L^2(s,T;U)} \qquad \forall q\in [s,T]\,,
\end{split}
\end{equation*}
where in the last estimate we used once again \cite[Theorem~383]{hardy-etal}.
This established $H_s^* z\in L^\infty(s,T;U)$.
That this regularity (in time) can be enhanced to $H_s^* z\in C([s,T];U)$ can be shown in the absence
of particular challenges and hence the proof is omitted.

\textbf{4.}
It is readily seen that 
\begin{equation*}
\cK_s^* \xi(q)=\int_s^T\int_s^t B^* k(\sigma-q)^* e^{A^*(t-\sigma)}\xi(t)\,d\sigma\,dt\,, 
\qquad q\in [0,s]\,,
\end{equation*}
so that 
\begin{equation*}
\begin{split}
\|\cK_s^* \xi(q)\|_U 
&\le c\left(\int_s^T\Big[\int_s^t \frac{1}{(t-\sigma)^\gamma}
\|k(\sigma-q)^*\|_{\cL(\cH)}\,d\sigma\Big]^2\,dt\right)^{1/2}
\left(\int_s^T\|\xi(t\|_\cH^2\,dt\right)^{1/2}
\\[1mm]
&\le c\,\|k\|_{L^2(0,T;\cL(\cH))} \|\xi\|_{L^2(0,s;U)} \qquad \forall q\in [0,s]
\end{split}
\end{equation*}
and $\cK_s^* \xi\in L^\infty(0,s;U)$.
As above, the $L^\infty$-in time can be actually enhanced to $C^0$-in time regularity; the proof is omitted.
\end{proof}

On the basis of the regularity results \eqref{e:reg-H} and \eqref{e:reg-K} pertaining to the operators $H_s$ and $\cK_s$, combined with the first of the memberships \eqref{e:reg-L_s} for $L_s$, we see that the lower regularity (in time) of $L_s$ prevails. 


\begin{corollary} \label{c:mild-slns}
Let the Assumptions~\ref{a:ipo_0} on the operators $A$, $B$, $k(\cdot)$ hold true, with 
$\gamma>\frac{1}{2}$.
Then, the mild solutions $w(\cdot)$ to the Cauchy problems associated with the control system
\eqref{e:system-start} that correspond to initial data $X_0$ (at initial time $s\in [0,T)$) and 
control functions $u\in L^2(s,T;U)$, given by \eqref{e:mild-sln_s}, are such that 
\begin{equation} \label{e:reg-mild} 
w\in L^{\frac{2}{2\gamma-1}}(s,T;\cH)\,. 
\end{equation}

\end{corollary}


\section{Towards the feedback formula. Proof of statements S1., S2.~and S3.}
\label{s:feedback}
In this section we retrace the major steps leading to the closed-loop representation of the optimal control in its final form \eqref{e:feedback_2}, that displays the very same building-blocks $P_i$ ($i=0,1,2$) of the quadratic form which yields the optimal value of the functional (see \eqref{e:optimal-cost_1}), thereby establishing the statements S2. and S3. of Theorem~\ref{t:main}.
First, we derive certain formulas for the optimal pair $(\hat{u}(t,s,X_0),\hat{w}(t,s,X_0))$,
which follow from the optimality condition; these lead in particular to establish the statement S1. of Theorem \ref{t:main}. 
Owing to the unboundedness of the control operator $B$, there are suitable regularity properties of the key operators $\psi_i$ and $Z_i$, $i=1,2$ (involved in these formulas) that will be needed later; thus,
they are discussed and pinpointed here in Lemma~\ref{l:regularity}.
Some of the proofs are essentially the same as the ones of analogous results in \cite[Section\,2]{ac-bu-JOTA} instead: when it will be that case, they will be omitted.
 
\subsection{The optimality condition.}
An easy computation provides a first rewriting of the cost functional; see also 
\cite[Section\,2.1]{ac-bu-JOTA}. 
This is a very first standard step in the theory of the LQ problem in the memoryless case, even in the presence of non coercive functionals.


\begin{lemma}
We make reference to the optimal control problem \eqref{e:mild-sln_s}-\eqref{e:cost_s}, under the standing Assumptions~\ref{a:ipo_0} and \eqref{e:ipo_1}.
Given $X_0\in Y_s$, the optimal cost admits the representation
\begin{equation} \label{e:q-form}
J_s(u,X_0)= \big(M_s X_0,X_0\big)_{Y_s} + 2 \text{Re}\, \big(N_s X_0,u\big)_{L^2(s,T;U)}
+ \big(\Lambda_s u,u\big)_{L^2(s,T;U)}\,,
\end{equation}
with
\begin{equation} \label{e:maiuscoli}
\begin{split}
\big(M_s X_0,X_0\big)_{Y_s}
&:=\int_s^T \big(C^*C E(t,s)X_0,E(t,s)X_0\big)_{\cH}\,dt
\\[1mm]
\big[N_s X_0\big](\cdot)&:=\big[\big(L_s^*+H_s^*\big) C^*C E(\cdot,s)X_0\big](\cdot)
\\[1mm]
\Lambda_s & :=I+\big(L_s^*+H_s^*\big)C^*C\big(L_s+H_s\big)\,,
\end{split}
\end{equation}
and where $I$ denotes the identity operator on $L^2(s,T;U)$.
The operator $\Lambda_s$ is boundedly invertible in $L^2(s,T;U)$, with $\|\Lambda_s^{-1}\|\le 1$ for all $s$. 
\end{lemma}
Having introduced the operators $M_s$, $N_s$, $\Lambda_s$ above, the optimality condition 
\begin{equation*}
\Lambda_s \hat{u}+N_s X_0=0
\end{equation*}
brings about the following result.


\begin{proposition}
We make reference to the optimal control problem \eqref{e:mild-sln_s}-\eqref{e:cost_s}, under the standing Assumptions~\ref{a:ipo_0} and \eqref{e:ipo_1}.
Given $X_0\in Y_s$, the unique optimal pair $(\hat{u}(t,s,X_0),\hat{w}(t,s,X_0))$ admits the following
representation:
\begin{equation} \label{e:opt-control}
\begin{split}
\hat{u}(t,s,X_0)&= \psi_1(t,s)w_0+ \int_0^s \psi_2(t,p,s)\eta(p)\,dp\,,
\\
\hat{w}(t,s,X_0) &=Z_1(t,s)w_0+ \int_0^s Z_2(t,p,s)\eta(p)\,dp\,,
\end{split}
\end{equation}
where we have set
\begin{align}
\psi_1(t,s)&:= -\Lambda_s^{-1} \big[\big(L_s^*+H_s^*\big) C^*C e^{A(\cdot-s)}\big](t)
\label{e:psi_1} 
\\[1mm]
\psi_2(t,p,s)&:= -\big[\Lambda_s^{-1} \big(L_s^*+H_s^*\big) C^*C \lambda(\cdot,p,s)\big](t)
\label{e:psi_2}
\\[2mm]
Z_1(t,s)&:= e^{A(t-s)}+\big[\big(L_s+H_s\big)\psi_1(\cdot,s)\big](t)\,,
\label{e:Z_1}
\\[1mm]
Z_2(t,p,s)&:= \lambda(t,p,s)+[(L_s+H_s)\psi_2(\cdot,p,s)](t).
\label{e:Z_2}
\end{align}

\end{proposition}

In the following lemma we pinpoint the regularity of the operators $Z_1(t,s)$ and $Z_2(t,p,s)$ with respect to the first variable. 

\begin{lemma} \label{l:regularity}
Let $Z_1(t,s)$ and $Z_2(t,p,s)$ be the operators defined in \eqref{e:Z_1} and \eqref{e:Z_2}, respectively.
Then, we have
\begin{equation} \label{e:Z-reg}
Z_1(\cdot,s)\in C([s,T],\cL(\cH))\,,
\quad
Z_2(\cdot,p,s)\in L^{2/(2\gamma-1)}(s,T;\cL(U,\cH))\,. 
\end{equation} 
\end{lemma}

\begin{proof}
From the definition \eqref{e:Z_1} of $Z_1$ we see that its regularity strictly depends on the one of 
$\psi_1(\cdot,s)$, so we discuss this one first.
Starting from the definition \eqref{e:psi_1} of $\psi_1$, we recall the well known property 
\begin{equation*}
L_s^*\in \cL(L^\infty(s,T;\cL(\cH)),C([s,T],\cL(U)))
\end{equation*}
recorded in the Appendix (see \eqref{e:L*}), along with the regularity pertaining to $H_s^*$ estabished in Proposition~\ref{e:regularity-newop} to infer $\psi_1(\cdot,s)\in C([s,T],\cL(U,\cH))$.
Then, in view of Proposition~\ref{p:Ls} and Proposition~\ref{e:regularity-newop} the very same membership for $Z_1(\cdot,s)$ holds true.

Similarly, we start from the definition of $Z_2(\cdot,p,s)$ and see that its regularity is determined by the ones of $\lambda(\cdot,p,s)$ and of $\psi_2(\cdot,p,s)$.
In this case we combine $\lambda(\cdot,p,s)\in L^{2/(2\gamma-1)}(s,T;\cL(U,\cH))$ -- that is  the regularity result \eqref{e:stima1} of Lemma~\ref{l:lemma1} --  with 
\begin{equation*}
\psi_2(\cdot,p,s)\in 
\begin{cases}
L^{\frac{2}{4\gamma-3}}(s,T;\cL(U,\cH)) &  \quad \gamma > 3/4 
\\
L^q(s,t;\cL(U,\cH)) \; \forall q<\infty &   \quad \gamma=3/4 
\\
C([s,T],\cL(U,\cH)) &  \quad \gamma < 3/4 
\end{cases}
\end{equation*}
-- coming from \eqref{e:L*} -- to find that
\begin{equation*}
\big(L_s+H_s\big) \psi_2(\cdot,p,s)\in \begin{cases}
L^{\frac{2}{6\gamma-5}}(s,T;\cL(U,\cH)) & \quad \gamma>5/6 \\
L^q(s,T;\cL(U,\cH)) \; \forall q<\infty & \quad  \gamma=5/6 \\
C([s,T],\cL(U,\cH)) & \quad \gamma < 5/6. 
\end{cases}
\end{equation*}

The above shows that $Z_2(\cdot,p,s)$ inherits the very same regularity of $\lambda(\cdot,p,s)$, 
i.e. the latter membership in \eqref{e:Z-reg}, thus concluding the proof.
\end{proof} 

In view of Lemma~\ref{l:regularity}, we infer the statememt S1. of Theorem~\ref{t:main}. 

\begin{corollary} \label{c:reg_coppia} 
Under the standing Assumption~\ref{a:ipo_0} and \eqref{e:ipo_1}, for every $X_0\in Y_s$ 
the we have the regularity in time of the optimal control $\hat{u}(t,s,X_0)$ and the component
$\hat{w}(t,s,X_0)$ asserted in \eqref{e:better-reg}. 
\end{corollary}

\begin{proof}  
We use the first of the representation formulas \eqref{e:opt-control}: as shown in the proof of Lemma~\ref{l:regularity}, we have $\psi_1(\cdot,s)\in C([s,T],\cL(U,\cH))$; 
moreover, by \eqref{e:psi_2} with $\theta\in (0,1-\gamma)$,
\begin{equation*} 
\|Z_2(t,\cdot,s)-Z_2(\tau,\cdot,s)\|_{L^2(0,s;\cL(U,\cH))} \le C\|\lambda(t,\cdot,s)-\lambda(\tau,\cdot,s)\|_{L^2(0,s;\cL(U,\cH))} \le C(t-\tau)^\theta\,,
\end{equation*}
which implies readily $\hat{u}(\cdot,s,X_0)\in C([s,T],U)$. 

Next, using the basic relation \eqref{e:mild-sln_s}, as well as \eqref{e:reg-K}, we immediately obtain
\begin{equation*}
\hat{w}(\cdot,s,X_0) = e^{A(\cdot-s)}w_0 + [(L_s+H_s)\hat{u}(\cdot,s,X_0)]+\cK_s\eta \in C([s,T],\cH).
\end{equation*}
\end{proof}

We need now appropriate estimates of the increments of $Z_1$ and $Z_2$ with respect to their other variables. 
\begin{lemma} \label{l:increm_Z}
Let $Z_1(t,s)$ and $Z_2(t,p,s)$ be the operators defined in \eqref{e:Z_1} and \eqref{e:Z_2}, respectively.
Then, we have for $t\ge\sigma\ge s$
\begin{equation} \label{e:Z-reg1}
\|Z_1(t,\sigma)-Z_1(t,s)\|_{\cL(\cH)} \le \omega(\sigma-s);
\end{equation}
furthermore, for $s\ge p >q$
\begin{equation} \label{e:Z-reg21}
\|Z_2(\cdot,p,s)-Z_2(\cdot,q,s)\|_{L^2(s,T;\cL(U))}\le \omega(p-q),
\end{equation}
and for $s >r\ge p$
\begin{equation}\label{e:Z-reg22}
\|Z_2(\cdot,p,s)-Z_2(\cdot,p,r)\|_{L^2(s,T;\cL(U))}\le \omega(s-r);
\end{equation}
here, $\omega$ is a continuous, positive function on $(0,T]$, such that 
$\omega(\tau)\to 0$ as $\tau \to 0$.  
\end{lemma}

\begin{proof}
The proof of \eqref{e:Z-reg1} is a straightforward consequence of the 
definition \eqref{e:Z_1} of $Z_1$ and of the fact that all the operators involved in its definition clearly depend continuously on $s$.\\
Concerning \eqref{e:Z-reg21}, starting from \eqref{e:Z_2}, for $t\ge s\ge r\ge p$ we have
\begin{equation}
\begin{split}
& \|Z_2(t,p,s)-Z_2(t,q,s)\|_{L^2(s,T;\cL(U,\cH))} \le \|\lambda(t,p,s)-\lambda(t,q,s)\|_{L^2(s,T;\cL(U,\cH))}
\\
& \qquad +\Big\|\Big[\big(L_s+H_s\big)\Lambda_s^{-1}\big(L_s^*+H_s^*\big) C^*C \big[\lambda(\cdot,p,s)-\lambda(\cdot,q,s)\big]\Big]\Big\|_{L^2(s,T;\cL(U,\cH))} 
\\
& \myspace \le C\big\|\lambda(\cdot,p,s)-\lambda(\cdot,q,s)\big\|_{L^2(s,T;\cL(U,\cH))}. 
\end{split}
\end{equation}
By \eqref{e:diff_lambda2}, we establish \eqref{e:Z-reg21}. \\
To prove the estimate \eqref{e:Z-reg22}, we compute for $t\ge s >r\ge p$:
\begin{equation*}
\begin{split}
& Z_2(t,p,s)-Z_2(t,p,r) = \big[\lambda(t,p,s)-\lambda(t,p,r)\big] 
\\
& \qquad + \big[\big(L_s+H_s\big)\Lambda_s^{-1}\big(L_s^*+H_s^*\big) C^*C \lambda(\cdot,p,s)\big](t)
\\
& \qquad -\big[\big(L_r+H_r\big)\Lambda_r^{-1}\big(L_r^*+H_r^*\big) C^*C \lambda(\cdot,p,r)\big](t)\,; 
\end{split}
\end{equation*}
then, as $L_s^*+H_s^*$ does not depend on $s$, we have
\begin{equation*}
\begin{split}
& \|Z_2(t,p,s)-Z_1(t,p,r)\|_{L^2(s,T;\cL(U,\cH))} \le 
\|\lambda(\cdot,p,s)-\lambda(\cdot,p,r)\|_{L^2(s,T;\cL(U,\cH))}
\\[1mm]
& \qquad + \big\|\big(L_s+H_s\big)\Lambda_s^{-1}\big(L_s^*+H_s^*\big) 
C^*C \big[\lambda(\cdot,p,s)-\lambda(\cdot,p,r)\big]\big\|_{L^2(s,T;\cL(U,\cH))} 
\\[1mm]
& \qquad + \big\|\big[\big(L_s+H_s\big)\Lambda_s^{-1}-\big(L_r+H_r\big)\Lambda_r^{-1}\big]
\big(L_r^*+H_r^*\big) C^*C \lambda(\cdot,p,r)\big\|_{L^2(s,T;\cL(U,\cH))}
\\[1mm]
& \quad =: s_1+s_2+s_3\,. 
\end{split}
\end{equation*}
The summand $s_1$ is estimated by \eqref{e:diff_lambda3}, while 
 $s_2$ can be bounded by a constant times $s_1$; thus, it remains to estimate the third summand $s_3$, for which a further decomposition leads to
\begin{equation*}
\begin{split}
s_3 & \le \big\|\big(L_s+H_s\big)\Lambda_s^{-1}\big[\Lambda_r-\Lambda_s\big]\Lambda_r^{-1}\big(L_r^*+H_r^*\big) C^*C
\lambda(\cdot,p,r)\big\|_{L^2(s,T;\cL(U,\cH))}
\\
& \ \ + \big\|\big[\big(L_s+H_s\big)-\big(L_r+H_r\big)\big]\Lambda_r^{-1}
\big(L_r^*+H_r^*\big) C^*C \lambda(\cdot,p,r)\big\|_{L^2(s,T;\cL(U,\cH))}.
\end{split}
\end{equation*}
It can be easily seen that the quantities in the right member behave all as $\omega(s-r)$, with $\omega(\tau)\to 0$ as $\tau \to 0$.  This proves  \eqref{e:Z-reg22}.
\end{proof}


\subsection{The optimal cost operators.}
This section is devoted to prove that a representation of the optimal control in closed-loop form
can be obtained, with gains that involve the very same linear bounded operators that are building blocks
of the optimal cost.
The line of argument is essentially the one pursued in \cite[Section~2]{ac-bu-JOTA}, whose trickier step 
was to develop different representations for the operators that are displayed in the optimal cost;
these allow to identify the presence of the optimal cost operators in a first feedback formula
that follows from the optimality condition (see Lemma~\ref{l:key} below). 

Since $B$ is unbounded, novel regularity results are called for: in particular, to give a meaning to the gain operators. 
These are provided later in Proposition~\ref{p:riccati-ops-continui} and Proposition~\ref{p:bounded-gains}, which in turn exploit the novel Lemma \ref{l:regularity} and Lemma \ref{l:increm_Z}.

We begin by providing a more precise and explicit formulation of the statement S2. of Theorem~\ref{t:main}.

\begin{proposition} \label{p:riccati-ops_0}
We make reference to the optimal control problem \eqref{e:mild-sln_s}-\eqref{e:cost_s}, under the standing Assumptions~\ref{a:ipo_0} and \eqref{e:ipo_1}.
Given $Y_0\in Y_s$, the optimal cost is a quadratic form in $Y_s$, which reads as
\begin{equation} \label{e:form}
\begin{split}
J_s(\hat{u})=J_s(\hat{u},X_0)
&=\big(P_0(s) w_0,w_0\big)_\cH
+ 2 \text{Re}\Big(\int_0^s P_1(s,p)\eta(p)\,dp,w_0\Big)_\cH
\\
& \myspace + \int_0^s\int_0^s \big(P_2(s,p,q)\eta(p),\eta(q)\big)_U\,dp\,dq\,,
\end{split}
\end{equation}
where the three operators $P_i$, $i=0,1,2$, are given by 
\begin{subequations} \label{e:riccati-ops}
\begin{align}
P_0(s)&=\int_s^T \big[Z_1(t,s)^*C^*C Z_1(t,s)+\psi_1(t,s)^*\psi_1(t,s)\big]\,dt\,,
\label{e:P_0_v1}
\\[1mm]
P_1(s,p)&=\int_s^T \big[Z_1(t,s)^*C^*C Z_2(t,p,s)+\psi_1(t,s)^*\psi_2(t,p,s)\big]\,dt\,,
\label{e:P_1_v1}
\\[1mm]
P_2(s,p,q)&= \int_s^T \big[Z_2(t,q,s)^*C^*C Z_2(t,p,s)+\psi_2(t,q,s)^*\psi_2(t,p,s)\big]\,dt\,,
\label{e:P_2_v1}
\end{align}
\end{subequations} 

\end{proposition}

\medskip
The presence of the operators $P_0$, $P_1$ and $P_2$ (defined by \eqref{e:riccati-ops} and which 
are building-blocks of the quadratic form \eqref{e:form}) is not immediately apparent in a first formula for the optimal strategy that follows in a first step combining the optimality condition with the transition properties of the optimal pair.
Pinpointing this fact requires that we deduce a suitable distinct representation of each operator
$P_i$, $i\in \{0,1,2\}$, beforehand.


\begin{lemma} \label{l:key}
With the operators $\psi_1(t,s)$ and $\psi_2(t,p,s)$ defined in \eqref{e:psi_1} and \eqref{e:psi_2}, $Z_1(t,s)$ and $Z_2(t,p,s)$ defined in \eqref{e:Z_1} and \eqref{e:Z_2}, respectively, then the optimal cost operators $P_i$ in \eqref{e:riccati-ops} can be equivalently rewritten as follows, $i\in \{0,1,2\}$:
\begin{subequations} \label{e:riccati-ops_2}
\begin{align}
P_0(t)&=\int_t^T e^{A^*(\sigma-t)}C^*CZ_1(\sigma,t)\,d\sigma\,,
\label{e:P_0_v2}
\\[1mm]
P_1(t,p)&=\int_t^T e^{A^*(\sigma-t)}C^*C Z_2(\sigma,p,t)\,d\sigma\,,
\label{e:P_1_v2}
\\[1mm]
P_2(t,p,q)&=\int_t^T \lambda(\sigma,q,t)^*C^*C Z_2(\sigma,p,t)\,d\sigma\,.
\label{e:P_2_v2}
\end{align}
\end{subequations}
Moreover, an additional (third) expression of $P_1(t,p)$ holds true:
\begin{equation} \label{e:P_1_v3}
P_1(t,p)=\int_t^T  Z_1(\sigma,t)^*\,C^*C\lambda(\sigma,p,t)\,d\sigma\,,
\end{equation}
so that in particular
\begin{equation} \label{e:P_1^*(t,t)}
P_1(t,t)^* = \int_t^T \lambda(\sigma,t,t)^*\,C^*C Z_1(\sigma,t)\,d\sigma\,.
\end{equation}

\end{lemma}

\medskip
In the following result we pinpoint the continuity of the operators $P_0(t)$, $P_1(t,p)$ and $P_2(t,p,q)$ with respect to the various independent variables.
The said regularity allows the integrals of $P_0(t)$, $P_1(r,r)$ and $P_2(r,p,r)$ as well as to use the bounds $\|P_0\|_\infty$, $\|P_1\|_\infty$ and $\|P_2\|_\infty$. 

\begin{proposition} \label{p:riccati-ops-continui}
The optimal cost operators $P_0(t)\in \cL(\cH)$, $P_1(t,p)\in \cL(U,\cH)$, $P_2(t,p,q)\in \cL(U)$ are continuous with respect to all the variables at hand.
\end{proposition}

\begin{proof}
\textbf{0.}  From the representation \eqref{e:P_0_v2} of $P_0$ we immediately obtain
\begin{equation*}
\|P_0(t)x\|_{\cH} \le M\int_t^T e^{\omega(\sigma-t)}\|Z_1(\sigma,t)x\|_{\cH}\,d\sigma 
\le c\|Z_1(\cdot,t)x\|_{\cH}\le c\|x\|_{\cH}\,,
\end{equation*}
so that
\begin{equation}\label{e:stima_unifP0}
\|P_0(t)\|_{\cL(\cH)}\le C, \qquad  0\le t \le T\,.
\end{equation}
An inspection of \eqref{e:P_0_v2} shows that in fact $P_0(\cdot)\in C([0,T],\cL(\cH))$.

\smallskip
\noindent
\textbf{1.} The formula \eqref{e:P_1_v2} for $P_1(t,p)$ yields the estimates
\begin{equation*}
\begin{split}
\|P_1(t,p)v\|_{\cH}
&\le 
M \int_t^T e^{\omega(\sigma-t)} \|Z_2(\sigma,p,t)\|_{\cL(U,\cH)}\|v\|_{U}\,d\sigma
\\
& \le M \,\Big(\int_t^T \|Z_2(\sigma,p,t)\|_{\cL(U,\cH)}^r \,d\sigma\Big)^{1/r} 
\Big(\int_t^T e^{\omega r'(\sigma-t)}\,d\sigma\Big)^{1/r'} \|v\|_{U} 
\\
& \le c \,\|v\|_{U} \big(e^{\omega r'(T-t)}-1\big)^{1/r'}, \quad  \textrm{where} \ r=\frac{2}{2\gamma-1}, \; r'=\frac{2}{3-2\gamma}\,,
\end{split}
\end{equation*}
with the latter implying
\begin{equation}\label{e:stima_unifP1}
\|P_1(t,p)\|_{\cL(U,\cH)}\le C, \qquad  0\le p \le t \le T\,.
\end{equation}
More specifically, the asymptotic estimate
\begin{equation}
\|P_1(t,p)\|_{\cL(\cH)}\le c \,\omega(T-t) \quad \textrm{with $\omega(\tau)\to 0$, as $\tau\to 0$} 
\end{equation}
holds true.

In order to prove continuity with respect to the first variable $t$, we compute for $t\ge \tau\ge p$
\begin{equation*}
\begin{split}
P_1(t,p)-P_1(\tau,p)&=\int_t^T \big[e^{A^*(\sigma-t)}-e^{A^*(\sigma-\tau)}\big] 
C^*C Z_2(\sigma,p,t)\,d\sigma
\\
& \qquad + \int_t^T e^{A^*(\sigma-\tau)}C^*C\big[Z_2(\sigma,p,t)-Z_2(\sigma,p,\tau)\big]\,d\sigma
\\
& \qquad + \int_t^T e^{A^*(\sigma-\tau)}C^*C Z_2(\sigma,p,\tau)\,d\sigma\,.
\end{split}
\end{equation*}
By \eqref{e:Z-reg22} and \eqref{e:Z-reg} we see that the increment $P_1(t,p)-P_1(\tau,p)$ tends to $0$ as $t-\tau \to 0$. 

As for the increment on the second independent variable, we see that 
\begin{equation*}
P_1(t,p)-P_1(t,q)=\int_t^T e^{A^*(\sigma-t)} C^*C \big[Z_2(\sigma,p,t)-Z_2(\sigma,q,t)\big]\,d\sigma
\end{equation*}
so that in this case we use \eqref{e:Z-reg21}, to find 
$\|P_1(t,p)-P_1(t,q)\|_{\cL(U,\cH)}\to 0$ as $p-q\to 0$.

\smallskip
\noindent
\textbf{2.} 
Starting from formula \eqref{e:P_2_v2} for $P_2(t,p,q)$, we find immediately that
\begin{equation*}
\|P_2(t,p,q)\|_{\cL(U)}\le 
c\,\Big[\int_t^T \|\lambda(\sigma,q,t)\|_{\cL(U,\cH)}^2\,d\sigma\Big]^{1/2}
\Big[\int_t^T\|Z_2(\sigma,p,t)\|_{\cL(U,\cH)}^2\,d\sigma\Big]^{1/2},
\end{equation*}
so that the uniform bound 
\begin{equation}\label{e:stima_unifP2}
\|P_2(t,p,q)\|_{\cL(U})\le C \qquad \textrm{for } t\ge q \ge p
\end{equation}
holds true.
Next, we evaluate (for $t\ge \tau\ge p$ and $\tau\ge q$)
\begin{equation*}
\begin{split}
P_2(t,p,q)-P_2(\tau,p,q)&=\int_t^T \big[\lambda(\sigma,q,t)^*-\lambda(\sigma,q,\tau)^*\big] 
C^*C Z_2(\sigma,p,t)\,d\sigma
\\
& \qquad + \int_t^T \lambda(\sigma,q,\tau)^*C^*C\big[Z_2(\sigma,p,t)-Z_2(\sigma,p,\tau)\big]\,d\sigma
\\
& \qquad + \int_t^T \lambda(\sigma,q,\tau)^* C^*C Z_2(\sigma,p,\tau)\,d\sigma\,.
\end{split}
\end{equation*}

By \eqref{e:stima1} and \eqref{e:diff_lambda3}, as well as \eqref{e:Z-reg} and \eqref{e:Z-reg22},  we deduce that $\|P_2(t,p,q)-P_2(\tau,p,q)\|_{\cL(U)}\to 0$ as $t-\tau\to 0$.

As for the increment on the second independent variable, we have
\begin{equation*}
P_2(t,p,q)-P_2(t,r,q)=\int_t^T \lambda(\sigma,q,t)^* C^*C \big[Z_2(\sigma,p,t)-Z_2(\sigma,r,t)\big]\,d\sigma
\end{equation*}
for $t\ge p\ge r$ and $t\ge q$, and again by \eqref{e:stima1} and \eqref{e:Z-reg21} we find $\|P_2(t,p,q)-P_2(t,r,q)\|_{\cL(U)}\to 0$ as $p-r\to 0$.
 
Finally, for $t\ge p$ and $t\ge q\ge r$ the increment on the third variable yields 
\begin{equation*}
P_2(t,p,q)-P_2(t,p,r)=\int_t^T \big[\lambda(\sigma,q,t)^*-\lambda(\sigma,r,t)^*\big]
C^*C Z_2(\sigma,p,t)\,d\sigma\,,
\end{equation*}
whose $\cL(U)$-norm tends to $0$, as ${q-r}$ tends to $0$, in view of \eqref{e:diff_lambda2} and \eqref{e:Z-reg}.
\end{proof}


\subsection{Closed loop optimal solution, the gain operators}
The analysis performed in the preceding section constitute the premise for the derivation of the
feedback formula \eqref{e:feedback_2}, which expresses the optimal control at time $t> s\ge 0$ in terms of the dynamics -- namely, the $w$-component of the state -- (pointwise in time) as well as of the past values of the optimal solution itself from $s$ forward, up to time $t$.
The statement S3.~of Theorem~\ref{t:main} is explicitly recorded in the result that follows
for the reader's convenience.

We emphasize that while the proof of Proposition~\ref{p:feedback-formula} below is akin to the one that led to show Proposition~1 in \cite{ac-bu-JOTA}, the unboundedness of the control operator $B$ raises the technical issue here of whether the gain operators $B^*P_0(t)$ and $B^*P_1(t,p)$ are defined also on elements of the respective functional spaces, and not just on the optimal evolution.
This question is answered positively at the end of the section; see the next Proposition~\ref{p:bounded-gains}.  
 
\begin{proposition} \label{p:feedback-formula}
Let $\hat{u}(t,s;X_0)$ be the optimal control for the minimization problem \eqref{e:mild-sln_s}-\eqref{e:cost_s}, with initial state $X_0$.
Then, the optimal control $\hat{u}$ admits the following representation: 
\begin{equation*} 
\begin{split}
\hat{u}(t,s,X_0)&=-\big[B^*P_0(t)+P_1(t,t)^*\big]\hat{w}(t,s,X_0)
\\
& \myspace -\int_0^t \big[B^*P_1(t,p)+P_2(t,p,t)\big]\theta(p)\,dp\,, 
\end{split}
\end{equation*}
with 
\begin{equation*}
\theta(\cdot)=
\begin{cases} \eta(\cdot) & \text{in $[0,s)$}
\\
\hat{u}(\cdot,s,X_0) & \text{in $[s,t)$}
\end{cases}
\end{equation*}
and the operators $P_i$ as in \eqref{e:riccati-ops_2} (originally, in \eqref{e:riccati-ops}), 
$i\in \{0,1,2\}$. 

\end{proposition}


\begin{proposition} \label{p:bounded-gains}
The (gain) operators $B^*P_0(t)$ and $B^*P_1(t,p)$ belong to the spaces $\cL(\cH)$ and $\cL(\cH,U))$ respectively. Moreover $B^*P_0(\cdot)\in C([0,T],\cL(\cH))$ and, for $s>p$,  
\begin{equation*}
B^*P_1(\cdot,p)\in L^\rho(s,T;\cL(\cH;U)) \quad \textrm{with } \rho \ 
\begin{cases} 
=\frac{2}{4\gamma-3} & \gamma>3/4
\\
<+\infty &  \gamma=3/4
\\
=+\infty & \gamma<3/4\,.
\end{cases}
\end{equation*}

\end{proposition}

\begin{proof}
Showing that the gain operator $B^*P_0(t)$ is bounded from $\cH$ into $U$ is pretty straightforward
(and also expected, as this result needs to be consistent with the memoryless case):
it suffices to write
\begin{equation*}
B^*P_0(t) = B^*(A^*)^{-\gamma}\int_t^T {A^*}^\gamma e^{A^*(\sigma-t)} C^*C Z_1(\sigma,t)\,d\sigma
\end{equation*}
and recall that $Z_1(\cdot,t)\in C([t,T],\cL(\cH))$ to find
\begin{equation*}
\|B^*P_0(t)\|_{\cL(\cH,U)} \le c\int_t^T \frac{1}{(\sigma-t)^\gamma}\|Z_1(\cdot,t)\|_{\cL(\cH)}
\le c(T-t)^{1-\gamma}\|Z_1(\cdot,t)\|_{\cL(\cH)}\,.
\end{equation*}

As for the gain $B^*P_1(t,p)$, first of all we need to show that for any $v\in U$ we have $P(t,p)v\in \cD(B^*)$.
Indeed, with $z\in \cD(B)$ we see that
\begin{equation*}
\begin{split}
(P_1(t,p)v,Bz)_{\cH} &=\Big(\int_t^Te^{A^*(\sigma-t)}C^*C Z_2(\sigma,p,t)v\,d\sigma,Bz\Big)_{\cH}
\\
& =\int_t^T \big(e^{A^*(\sigma-t)}C^*C Z_2(\sigma,p,t)v,Bz\big)_{\cH}\,d\sigma
\\
& =\int_t^T \big((B^*{A^*}^{-\gamma})\,{A^*}^\gamma e^{A^*(\sigma-t)}C^*C Z_2(\sigma,p,t)v,z\big)_U\,d\sigma
\end{split}
\end{equation*}
which yields 
\begin{equation*}
(P_1(t,p)v,Bz)_{\cH}\le c\int_t^T \frac{1}{(\sigma-t)^\gamma} \|Z_2(\sigma,p,t)v\|_{\cH}\,d\sigma \|z\|_U
\end{equation*}
so that 
\begin{equation*}
\|B^*P_1(t,p)v\|_{\cH}\le c\int_t^T \frac{1}{(\sigma-t)^\gamma} \|Z_2(\sigma,p,t)v\|_{\cH}\,d\sigma\,.
\end{equation*}
Thus, since (by \eqref{e:Z-reg}) $Z_2(\cdot,p,t)\in L^r(t,T;\cL(U,\cH))$ with $r=2/(2\gamma-1)$, it follows
\begin{equation}\label{e:reg_B*P1}
B^*P_1(\cdot,p)v\in L^\rho(t,T;U)\,, \; \textrm{with} \; \rho
\begin{cases} 
=\frac{2}{4\gamma-3} & \gamma>\frac{3}{4}
\\
<+\infty &  \gamma=\frac{3}{4}
\\
=+\infty & \gamma<\frac{3}{4}\,,
\end{cases}
\end{equation} 
along with 
\begin{equation*}
\|B^*P_1(\cdot,p)v\|_{L^\rho(t,T;U)}\le c \|Z_2(\cdot,p,t)v\|_{L^r(t,T;\cH)}\le c\|v\|_{U}
\end{equation*}
for every $t>p$. 
\end{proof}


\section{Well-posedness for the coupled system of quadratic operator equations. 
Proofs of the statements S4.-S5.}
\label{s:wellposedRE}
This section deals with the issue of the unique determination of the triplet of operators which 
enter the feedback representation \eqref{e:feedback_2} of the optimal control. 
This is a key step to achieve the optimal synthesis.
Showing that the operators $P_0(t)$, $P_1(t,p)$, $P_1(t,p)^*$, $P_2(t,p,q))$ do solve a certain coupled system of quadratic (operator) equations corresponding to the optimal control problem on $[s,T]$, that is \eqref{e:DRE}, establishes the property of existence for the system \eqref{e:DRE}.
The next and final step is to prove the property of uniqueness for the solutions to \eqref{e:DRE}, thus
confirming its well-posedness.

\smallskip
\paragraph{Existence.} 
For the proof of existence we omit the details and instead refer the reader to the necessarily 
long computations in \cite[Section~2]{ac-bu-JOTA} ({\sl cf.} the proof of statement S5.), which are
equally valid here.
We limit ourselves to state explictly a result which plays a primary role in the computation of the derivatives of $P_0$, $P_1$ and $P_2$.


\begin{lemma} \label{l:derivatives}
Let $\psi_1(p,t)$, $\psi_2(r,p,t)$, $Z_1(p,\sigma)$, $Z_2(\sigma,p,t)$ the operators
defined in \eqref{e:psi_1}, \eqref{e:psi_2}, \eqref{e:Z_1} and \eqref{e:Z_2} respectively.
If $x\in \cD(A)$ and $v\in U$, then the derivatives 
$\partial_t\psi_1(p,t)x$, $\partial_t\psi_2(r,p,t)v$, 
$\partial_t Z_1(p,\sigma)x$, $\partial_t Z_2(\sigma,p,t)v$ exist, with

\begin{subequations} \label{e:derivatives1}
\begin{align} 
\partial_t Z_1(\sigma,t)x 
&= -e^{A(\sigma-t)}Ax +\big[\big(L_t+H_t\big)\partial_t \psi_1(\cdot,t)x\big](\sigma)
\notag\\
& \qquad
- \big[e^{A(\sigma-t)}B+\lambda(\sigma,t,t)\big]\psi_1(t,t)x\in L^{\frac1{\gamma+\varepsilon}}(t,T;\cH)\,,
\label{e:D-Z_1}
\\[2mm]
\partial_t Z_2(\sigma,p,t)v &= 
-e^{A(\sigma-t)}k(t-p)Bv +\big[\big(L_t+H_t\big)\partial_t \psi_2(\cdot,p,t)v\big](\sigma)
\notag\\
& \qquad- \big[e^{A(\sigma-t)}B+\lambda(\sigma,p,t)\big]\psi_2(\sigma,p,t)v
\in L^{\frac1{\gamma+\varepsilon}}(t,T;\cH)\,,
\label{e:D-Z_2}
\end{align}
\end{subequations}
where
\begin{subequations} \label{e:derivatives2}
\begin{align} 
\partial_t \psi_1(p,t)x &= 
\Lambda_t^{-1}\Big[\big(L_t^*+H_t^*\big)C^*C\big[\big(e^{A(\cdot-t)}B+\lambda(\cdot,t,t)\big)\big]
\psi_1(t,t)x
\notag\\
& \qquad + e^{A(\cdot-t)}Ax\Big](p)\in L^r(t,T;U)\,,
\\[2mm]
\partial_t \psi_2(r,p,t)v &= 
\Lambda_t^{-1}\Big[\big(L_t^*+H_t^*\big)C^*C\big[\big(e^{A(\cdot-t)}B+\lambda(\cdot,t,t)\big)\big]\psi_2(t,p,t)v
\notag \\
& \qquad + e^{A(\cdot-t)}k(t-p)Bv\Big](r)\in L^r(t,T;U)
\end{align}
\end{subequations}
with $\varepsilon \in (0,1-\gamma)$ and for $r=\frac2{2\gamma-1}$.  
\end{lemma}

With the representations \eqref{e:riccati-ops_2} of $P_0(t)$, $P_1(t,p)$, $P_2(t,p,q)$ as a starting point, using the expressions \eqref{e:derivatives1} and \eqref{e:derivatives2} stated in Lemma~\ref{l:derivatives} -- along with the memberships therein -- we can compute the derivatives of
$(P_0(t)x,y)_{\cH}$, $(P_1(t,p)v,y)_{\cH}$ and $(P_2(t,p,q)u,v)_U$, with $x,y\in \cD(A)$ and $u,v\in U$, thereby attaining the differential system \eqref{e:DRE}.

\smallskip
\paragraph{Uniqueness.}
To confirm the property of uniqueness, we follow the line of argument pursued in 
\cite[Theorem 4.1]{ac-bu-JOTA}.
We rewrite the differential system \eqref{e:DRE} solved by $P_0(t)$, $P_1(t,p)$ and $P_2(t,p,q))$
as a matrix (operator) differential equation, that is
\begin{equation} \label{e:big-DRE}
\hspace{7mm}
\frac{d}{dt} P(t)=-\cQ -P(t)[\cA+\cK_1(t)] -[\cA^*+\cK_2(t)]P(t)
+ P(t)\cI_{1,t}\cB\cI_{2,t}P(t)\,,
\end{equation}
supplemented with the final condition $P(T)=0$, having set
\begin{equation} \label{e:big-P}
P(t)=\begin{pmatrix} 
P_0(t)  & P_1(t,\cdot)
\\
P_1(t,:)^* & P_2(t,\cdot,:)
\end{pmatrix}
\end{equation}
and 
\begin{equation*}
\begin{split}
& \cQ=\begin{pmatrix}
C^*C & 0 
\\
0 & 0
\end{pmatrix}, \quad 
\cK_1(r)=\begin{pmatrix} 
0 & k(r-\cdot)
\\
0 & 0
\end{pmatrix},\quad
\cK_2(r)=\begin{pmatrix} 
0 & 0
\\
k(r-:)^* & 0
\end{pmatrix},
\\[1mm]
& \quad \ \cI_{1,r}=\begin{pmatrix} 
I & 0
\\
0 & \chi_{\{r\}}(\cdot)
\end{pmatrix}, \quad
\cI_{2,r}=\begin{pmatrix} 
I & 0
\\
0 & \chi_{\{r\}}(:)
\end{pmatrix},\ \ 
\cB=\begin{pmatrix} 
BB^* & B
\\
B^* & I
\end{pmatrix}.
\end{split}  
\end{equation*}
The differential equation \eqref{e:big-DRE} in $Y_s$ can be equivalently written in its
integral form
\begin{equation} \label{e:big-IRE}
P(t)=\int_t^T e^{\cA^*(r-t)}\Big[\cQ +[P(r)\cK_1(r)+\cK_2(r)P(r) - P(r)\cI_{1,r}\cB \cI_{2,r} P(r)\Big]e^{\cA(r-t)}\,dr\,,
\end{equation}
having introduced the $C_0$-semigroup
\begin{equation*}
e^{\cA s}=\begin{pmatrix}
e^{As} & 0 
\\
0 & I
\end{pmatrix}, \qquad s\ge 0.
\end{equation*}

We know that $P(t)$ defined by \eqref{e:big-P} is a solution to \eqref{e:big-IRE}, so let us 
assume that $Q(t)$ is another solution: setting
\begin{equation*}
\begin{split}
& V(t)=P(t)-Q(t)
\\
& \qquad =\begin{pmatrix} 
P_0(t)-Q_0(t)  & P_1(t,\cdot)-Q_1(t,\cdot)
\\
P_1(t,:)^*-Q_1(t,:)^* & P_2(t,\cdot,:)-Q_2(t,\cdot,:)
\end{pmatrix} = \begin{pmatrix} 
V_0(t)  & V_1(t,\cdot)
\\
V_1(t,:)^* & V_2(t,\cdot,:)
\end{pmatrix},
\end{split}
\end{equation*}
we see that $V(t)$ solves the integral equation
\begin{equation} \label{e:mildV} 
\begin{split}
& V(t) = \int_t^T e^{\cA^*(r-t)}\Big[V(r)\cK_1(r)+\cK_2(r)V(r)
\\
& \myspace \qquad 
-V(r)\cI_{1,r}\cB \cI_{2,r}P(r) - Q(r)I_{1,r}\cB I_{2,r}V(r)\Big]e^{\cA(r-t)}\,dr\,.
\end{split}
\end{equation}

Take now $s\in [0,T)$. Our goal in what follows is to produce an {\em a priori} estimate for the quantity
\begin{equation}
\label{e:somma_norme}
\begin{split}
& \|V_0\|_{L^\infty(s,T;\cL(\cH))} + \|B^*V_0\|_{L^\infty(s,T;\cL(\cH,U))} + \sup_{p\in [0, t];\, t\in[s,T]} \|V_1(t,p)\|_{\cL(\cH)} 
\\
& + \sup_{p,q\in [0,t];\, t\in[s,T]}\|V_2(\cdot,p,q)\|_{\cL(U)}+\sup_{p\in [0,t]; \,t\in [s,T]} \|B^*V_1(\cdot,p)\|_{L^2(t,T;\cL(\cH,U))},
\end{split}
\end{equation}
which will be simply indicated as  
\begin{equation*}
A(s):=\|V_0\|_{\infty,s} + \|B^*V_0\|_{\infty,s} + \|V_1\|_{\infty,s} + \|V_2\|_{\infty,s} 
+ \|B^*V_1(\cdot,p)\|_{2,s}.
\end{equation*}
We also recall that by \eqref{e:reg_B*P1} we know that $B^*P_1(\cdot,p)$ and $B^*Q_1(\cdot,p)$ belong to $L^r(s,T;\cL(\cH,U)$, $r=\frac2{4\gamma-3}$. 

The four (scalar) integral equations to which \eqref{e:mildV} is equivalent allow to establish the following a priori estimates (with $t\in[s,T]$ and $p,q\in [0,t]$):
\begin{equation} 
\begin{split}
& \|V_0(t)\|_{\cL(\cH)} \le C \int_t^T \Big[\|B^*V_0(\sigma)\|_{\cL(\cH,U)} +\|V_1(\sigma,\sigma)\|_{\cL(\cH)} \Big]d\sigma 
\\
& \qquad \qquad \quad \le C(T-s) \Big[ \|B^*V_0\|_{\infty,s} + \|V_1\|_{\infty,s}\Big];
\end{split}
\end{equation}
\begin{equation} 
\begin{split}
& \|B^*V_0(t)\|_{\cL(\cH,U)} \le C \int_t^T \frac1{(\sigma-t)^\gamma}\Big[\|B^*V_0(\sigma)\|_{\cL(\cH,U)} +\|V_1(\sigma,\sigma)\|{\cL(\cH)}\Big]d\sigma
\\
& \qquad \qquad \qquad \quad \le C(T-s)^{1-\gamma} \Big[\|B^*V_0\|_{\infty,s} + \|V_1\|_{\infty,s}\Big];
\end{split}
\end{equation}
\begin{equation} 
\begin{split}
& \|V_1(t,p)\|_{\cL(\cH)} \le C\int_t^T \Big\{ \|k(\sigma-p)\|_{\cL(\cH)}\,\|V_0\|_{\infty,s} \\
& \qquad \qquad \qquad \qquad \quad + \Big[\|B^*V_0\|_{\infty,s} + \|V_1\|_{\infty,s}\big]\|B^*P_1(\sigma,p)\|_{\cL(\cH,U)}
\\
& \qquad \qquad \qquad \qquad \quad +\|B^*V_1(\sigma,p)\|_{\cL(\cH,U)}+ \|V_2\|_{\infty,s}\Big\}d\sigma
\\
& \qquad \qquad \qquad \le C(T-s)^{\frac12} \Big[\|V_0\|_{\infty,s}+\|B^*V_0\|_{\infty,s}-\|V_1\|_{\infty,s} +\|B^*V_1\|_{2,s}\Big]
\\
& \qquad \qquad \qquad \qquad \quad+ C(T-s)\|V_2\|_{\infty,s};  
\end{split}
\end{equation}
\begin{equation} 
\begin{split}
& \|V_2(t,p,q)\|_{\cL(U)} \le C\int_t^T \Big\{\Big[\|k(\sigma-p\|_{\cL(\cH)}+ \|k(\sigma-q\|_{\cL(\cH)}\Big] \|V_1\|_{\infty,s}
\\
& \qquad \qquad + \Big[\|B^*V_1(\sigma,q)\|_{\cL(\cH,U)}+\|V_2\|_{\infty,s}\Big]\Big[\|B^*P_1(\sigma,p)\|_{\cL(\cH,U)}+C\Big]
\\
& \qquad \qquad + \Big[\|B^*V_1(\sigma,p)\|_{\cL(\cH,U)}+\|V_2\|_{\infty,s}\Big]\Big[\|B^*Q_1(\sigma,q)\|_{\cL(\cH,U)}+C\Big]\Big\}d\sigma
\\
& \qquad\quad \le C(T-s)^{\frac12} \Big[\|V_1\|_{\infty,s} + \|V_2\|_{\infty,s}+ \|B^*V_1\|_{2,s}\Big];
\end{split}
\end{equation}
\begin{equation} 
\begin{split}
& \|B^*V_1(\cdot,p\|_{L^2(t,T;\cL(\cH,U))} \le C\|V_0\|_{\infty,s} \left\|\int_\cdot ^T \frac1{(\sigma-\cdot)^\gamma} \|k(\sigma-p)\|_{\cL(\cH)}d\sigma\right\|_{L^2(t,T)} 
\\
& \quad + \Big[\|B^*V_0\|_{\infty,s}+\|V_1\|_{\infty,s}\Big] \Bigg[\left\|\int_\cdot ^T \frac1{(\sigma-\cdot)^\gamma} \Big[\|B^*P_1(\sigma,p)\|_{\cL(\cH,U)} + C\Big]d\sigma\right\|_{L^2(t,T)}
\\
& \quad + \left\|\int_\cdot^T \frac1{(\sigma-\cdot)^\gamma} \big[\|B^*V_1(\sigma,p)\|_{\cL(\cH,U)}+\|V_2\|_{\infty,s}\big]d\sigma\right\|_{L^2(t,T)} 
\\
& \le C(T-s)^{1-\gamma} \Big[\|V_0\|_{\infty,s} +\|B^*V_0\|_{\infty,s} + \|V_1\|_{\infty,s}+\|B^*V_1\|_{2,s} +\|V_2\|_{\infty,s}\Big].
\end{split}
\end{equation}

Thus we return to \eqref{e:somma_norme} taking into account the five estimates above, to find
\begin{equation*}
A(s)\le C (T-s)^{1-\gamma} A(s)\,.
\end{equation*}
Then, if $T-s$ is sufficiently small, say $T-s<t_0$, we get $A(\sigma)=0$ in $[T-t_0,T]$. 
Repeating the above argument in $[s,T-t_0]$, we obtain $A(\sigma)=0$ in $[T-2t_0,T]$. 
In a finite number of steps we obtain $A(\sigma)=0$ in $[s,T]$; by definition, this means 
$V(t)=P(t)-Q(t)=0$ in $[s,T]$. 
Since $s$ was arbitrary, we deduce $P\equiv Q$, namely, the solution to the equation \eqref{e:mildV}
is unique.

{\small
\section*{Funding} 
The research of the second author has been performed in the framework of the MIUR-PRIN Grant 2020F3NCPX ``Mathematics for Industry 4.0 (Math4I4)''.
F.~Bucci was supported by the Universit\`a degli Studi di Firenze under the 2024 Project {\em ``Controllo lineare-quadratico per equazioni integro-differenziali a derivate parziali''}, of which she is responsible.
She is also a member of the Gruppo Nazionale per l'Analisi Mate\-ma\-tica, la Probabilit\`a  e le loro Applicazioni of the Istituto Nazionale di Alta Matematica (GNAMPA-INdAM) and participant to
the 2024 GNAMPA Project ``Controllo ottimo infinito dimensionale: aspetti teorici ed applicazioni''.
}


\appendix


\section{A few instrumental results}
In this Appendix we recall a few regularity results which are used in the paper.
These include results pertaining to the (time and space) regularity of the mapping $L_s$ and of its adjoint $L_s^*$ defined in \eqref{e:L_s} and \eqref{e:L_s*}, respectively; these are established outcomes in the context of {\em memoryless} control systems of the form $y'=Ay+Bu$ under the first two of the Assumptions~\ref{a:ipo_0} (a pattern which is
consistent with parabolic-like dynamics subject to boundary data/actions). 
Lastly, we record a more basic result pertaining to convolutions. 

The first result can be found in \cite[Vol.~I]{las-trig-redbooks}.

\begin{proposition} \label{p:Ls} 
Let $L_s$ and $L_s^*$ be the maps defined by \eqref{e:L_s} and \eqref{e:L_s*}, respectively.
Then,
\begin{equation} \label{e:reg-L_s}
L_s\colon L^2(s,T;U) \longrightarrow 
\begin{cases}
L^{\frac{2}{2\gamma-1}}(s,T;\cH) & \textrm{if $\gamma>\frac{1}{2}$}
\\[1mm]
\bigcap\limits_{q<\infty}L^q(s,T;\cH) & \textrm{if $\gamma=\frac{1}{2}$}
\\[1.3mm]
C([s,t],\cH) & \textrm{if $\gamma<\frac{1}{2}$}
\end{cases}
\end{equation}
continuously, and more generally
\begin{equation} \label{e:reg-L_s-general}
L_s\colon L^p(s,T;U) \longrightarrow 
\begin{cases}
L^{\frac{p}{1-(1-\gamma)p}}(s,T;\cH) & \textrm{if $p<\frac{1}{1-\gamma}$}
\\[1mm]
\bigcap\limits_{q<\infty}L^q(s,T;\cH) & \textrm{if $p=\frac{1}{1-\gamma}$}
\\[1.3mm]
C([s,t],\cH) & \textrm{if $p>\frac{1}{1-\gamma}$}
\end{cases}
\end{equation}
continuously.

For the adjoint operator, we have
\begin{equation} \label{e:L*}
L_s^*\colon L^p(s,T;\cH) \longrightarrow 
\begin{cases}
L^{\frac{p}{1-(1-\gamma)p}}(s,T;U) & \textrm{if $p<\frac{1}{1-\gamma}$}
\\[1mm]
\bigcap\limits_{q<\infty}L^q(s,T;U) & \textrm{if $p=\frac{1}{1-\gamma}$}
\\[1.3mm]
C([s,t],U) & \textrm{if $p>\frac{1}{1-\gamma}$}
\end{cases}
\end{equation}
continuously.

\end{proposition}

\smallskip

The next lemma is the classical Young inequality for convolutions, whose proof is found e.g. in \cite[p.~100]{lieb-loss}.

\begin{lemma} \label{l:young}
Let $f\in L^p(0,T-s)$ and $g\in L^q(s,T)$, with $\frac{1}{p}+\frac{1}{q}-1>0$.
Then, the convolution $f\ast g$ defined by 
\begin{equation} \label{e:convolution1}
(f\ast g)(t):=\int_s^t f(t-\sigma)g(\sigma)\,d\sigma \qquad t\in [s,T]
\end{equation}
belongs to $L^r(s,T)$ with $\frac{1}{r}=\frac{1}{p}+\frac{1}{q}-1$, and the following estimate
holds true
\begin{equation*}
\|f\ast g\|_{L^r(s,T)}\le \|f\|_{L^p(0,T-s)}\,\|g\|_{L^q(s,T)}\,. 
\end{equation*}

If $\frac{1}{p}+\frac{1}{q}-1<0$, then $f\ast g\in C([s,T])$, and we have
\begin{equation*}
\|f\ast g\|_{C([s,T])}\le \|f\|_{L^p(0,T-s)}\,\|g\|_{L^q(s,T)}\,. 
\end{equation*}
An analogous results holds in the case $f\in L^p(s-T,0)$, $g\in L^q(s,T)$, and with
\begin{equation*}
(f\ast g)(q):=\int_q^T f(q-\sigma)g(\sigma)\,d\sigma \qquad q\in [s,T]\,.
\end{equation*}

\end{lemma}


\end{document}